\documentclass[11pt,a4paper,leqno,twoside, headinclude]{amsart}

\usepackage[tracking=true]{microtype}
\DeclareMicrotypeSet*[tracking]{my}%
  { font = */*/*/sc/* }%
\SetTracking{ encoding = *, shape = sc }{ 45 }
\usepackage{multirow}
\usepackage{booktabs}
\usepackage{longtable}
\usepackage{lscape}
\usepackage{dynkin-diagrams}
\usepackage[tableposition=t]{caption}
\usepackage{arydshln}
\title[Submanifolds of symmetric spaces]{The flavour of intermediate Ricci and homotopy when studying submanifolds of symmetric spaces}
\def\titl{The flavour of intermediate Ricci and homotopy when studying submanifolds of symmetric spaces}
\def\auth{Manuel Amann \& Peter Quast \& Masoumeh Zarei}
\subjclass[2010]{53C35, 53C40, 53C20, 57T20 (Primary), 53C42 (Secondary)}
\keywords{\noindent symmetric space, submanifold, k-positive Ricci curvature, homotopy groups, Cartan type, totally geodesic}
\thanks{}
\author{\auth}

\usepackage[english]{babel}

\usepackage[colorlinks,pdftex, plainpages=false]{hyperref}
\hypersetup{pdftitle=\titl, pdfauthor=\auth, pdftoolbar=false,
plainpages=false, hyperindex=true, pdfdisplaydoctitle=true}
\usepackage{cleveref}
\hypersetup{colorlinks,%
citecolor=black,%
filecolor=black,%
linkcolor=black,%
urlcolor=black,%
pdftex}

\usepackage{color}

\usepackage{amsmath, amssymb, amscd, amsthm}
\usepackage{stmaryrd}
\usepackage{latexsym}
\usepackage[all]{xy}
\usepackage{pb-diagram}
\usepackage{rotating}
\usepackage{multicol}
\usepackage{lscape}
\usepackage{wasysym}
\usepackage{longtable}
\usepackage{enumerate}
\usepackage{threeparttable}
\usepackage[graphicx]{realboxes} %Rotating tables
\xyoption{all}

\newtheorem{theo}{Theorem}[section]
\newtheorem{main}{Theorem}
\newtheorem{maincor}[main]{Corollary}
\newtheorem*{main*}{Theorem}

\newtheorem*{mainprop*}{Proposition}

\newtheorem{mainconj}{Conjecture}

\newtheorem{prop}[theo]{Proposition}
\newtheorem{defi2}[theo]{Definition}
\newtheorem*{defi2*}{Definition}

\newenvironment{defi*}{\begin{defi2*}\normalfont}{\end{defi2*}}

\newenvironment{defin*}[1]{\begin{defi2*}[#1]\normalfont}{\end{defi2*}}
\newtheorem*{rem2*}{Remark}
\newenvironment{rem*}{\begin{rem2*}\normalfont}{\hfill$\boxbox$\end{rem2*}}
\newtheorem{rem2}[theo]{Remark}
\newenvironment{rem}{\begin{rem2}\normalfont}{\hfill$\boxbox$\end{rem2}}

\newtheorem{lemma}[theo]{Lemma}
\newtheorem{cor}[theo]{Corollary}
\newtheorem*{cor*}{Corollary}

\newtheorem*{conj*}{Conjecture}
\newtheorem*{theo*}{Theorem}
\newtheorem*{ques*}{Question}
\newtheorem*{mi2}{Main Idea}

\newtheorem{ex2}[theo]{Example}

\newtheorem{exer2}[theo]{Exercise}

\newtheorem{alg2}[theo]{Algorithm}

\newcommand{\cc}{{\mathbb{C}}} % complex numbers
\newcommand{\hh}{{\mathbb{H}}} % quaternions
\newcommand{\qq}{{\mathbb{Q}}} % rational numbers
\newcommand{\rr}{{\mathbb{R}}} % real numbers
\newcommand{\pp}{{\mathbf{P}}} % polynomials, projective space
\newcommand{\Gr}{{\mathbf{Gr}}} % Grassmannian
\newcommand{\V}{{\mathbf{V}}} % Stiefel
\newcommand{\s}{{\mathbb{S}}} % sphere
\newcommand{\zz}{{\mathbb{Z}}} % integers
\newcommand{\Spe}{{\mathbf{S}}} % Special
\newcommand{\SO}{{\mathbf{SO}}} % special orthogonal group
\newcommand{\U}{{\mathbf{U}}} % unitary group
\newcommand{\SU}{{\mathbf{SU}}} % special unitary group
\newcommand{\Sp}{{\mathbf{Sp}}} % group
\newcommand{\B}{{\mathbf{B}}} % Lie group type B
\newcommand{\C}{{\mathbf{C}}} % Lie group type C
\newcommand{\E}{{\mathbf{E}}} % Lie group type E
\newcommand{\F}{{\mathbf{F}}} % Lie group type F
\newcommand{\G}{{\mathbf{G}}} % Lie group type G
\newcommand{\Spin}{{\mathbf{Spin}}} % spin group
\newcommand{\dif} {{\operatorname{d}}} % differential operator d
\newcommand{\In} {{\,\subseteq\,}} % subset
\newcommand{\codim}{{\operatorname{codim\,}}} % codimension
\newcommand{\id}{{\operatorname{id}}} % identity
\newcommand{\rk}{{\operatorname{rk\,}}} % toral rank
\newcommand{\Ric}{{\operatorname{Ric}}} % Ricci tensor
\newcommand{\shape}{{\operatorname{S}}} %Shape operator
\newcommand{\foc}{{\operatorname{foc}}} % Focal radius
\newcommand{\co}{\colon\thinspace} % colon in maps
\newcommand{\comment}[1]{} % insert a large comment
\newcommand{\hto}[1]{\overset{#1}{\hookrightarrow}} % abbreviation for labelled \hookrightarrow

\newcommand{\step}[1]{\textbf{Step #1.}} % Step
\newcommand{\ack}{\noindent\textbf{Acknowledgements. }} % Acknowledgements
\newcommand{\str}{\noindent\textbf{Structure of the article. }} % Structure of the article
\newcommand{\mytableextraspace}{\addlinespace[.4em]} %For the table

\newcommand{\Codim}{{\mathrm{codim}}}

\newcommand{\lieA}{{\mathfrak{a}}}

\newcommand{\lieG}{{\mathfrak{g}}}

\newcommand{\lieK}{{\mathfrak{k}}}

\newcommand{\lieP}{{\mathfrak{p}}}

\newcommand{\lieO}{{\mathfrak o}}

\newcommand{\lieSU}{{\mathfrak{su}}}
\newcommand{\lieU}{{\mathfrak u}}
\newcommand{\lieSp}{{\mathfrak{sp}}}

\newcommand{\Root}{{\mathcal{R}}}
\newcommand{\trace}{{\mathrm{trace}}}

\newcommand{\lieE}{{\mathfrak{e}}}
\newcommand{\lieF}{{\mathfrak{f}}}

\newcommand{\Ad}{{\mathrm{Ad}}}

\def\N{{\mathbb N}}

\def\R{{\mathbb R}}
\def\C{{\mathbb C}}

\begin{document}

\maketitle \thispagestyle{empty}

%%%%%%%%%%%%%%%%%%%%%%%%%%%%%%%%%% Abstract%%%%%%%%%%%%%%%%%%%%%%%%%%%%%%%%%%%%%%%%

\begin{abstract}
We introduce a new technique to the study and identification of submanifolds of simply-connected symmetric spaces of compact type based upon an approach computing $k$-positive Ricci curvature of the ambient manifolds and using this information in order to determine how highly connected the embeddings are.

This provides codimension ranges in which the Cartan type of submanifolds satisfying certain conditions which generalize being totally geodesic necessarily equals the one of the ambient manifold. Using results by Guijarro--Wilhelm our approach partly generalizes recent work by Berndt--Olmos on the index conjecture.
\end{abstract}

%%%%%%%%%%%%%%%%%%%%%%%%%%Introduction%%%%%%%%%%%%%%%%%%%%%%%%%%%

\section*{Introduction}

Symmetric spaces are amongst the most beautiful structures visible in Riemannian geometry. They have undergone decades of study, and their structure theory which was established on this quest has incorporated many influences ranging from Riemannian geometry to pure algebra. The outcome is as fascinating as it is elegant.

Over the years trying to understand their properties even better one focus was lying on understanding their submanifolds. It seems however, that any information of this kind is hard to come by. There has been remarkable insight as for totally geodesic submanifolds in rank at most 3, respectively, in general, as finished recently by Berndt--Olmos in the \emph{index}, the lowest codimension of a totally geodesic submanifold. Information beyond the case of totally geodesic submanifolds either seems elusive, or, it appears that the number of submanifolds satisfying interesting geometric properties is more than abundant (as already is the case for minimal submanifolds). Indeed, the study of totally geodesic submanifolds relies heavily on the understanding of Lie triple systems, and any further technique seems highly desirable.

In the spirit of combining Riemannian geometry, topology and algebra in the area of symmetric spaces we draw on the ``generalized connectedness lemma'' by Guijarro--Wilhelm, inspired by the following proposition, in order to suggest such a new approach to the study of certain classes of submanifolds of symmetric spaces genuinely containing totally geodesic ones.

\begin{prop}\label{P:Ric_k}
Let $P$ be a symmetric space of compact type.
Then there exists a positive integer $k_P$ such that for all integers $k$ with $\dim P> k\geq k_P$ it holds that $\Ric_{k}>0$.
Moreover, if $\rk P=1$, then $k_P=1$, and if $\rk P\geq 2$ and $P$ is
additionally irreducible and simply-connected, then $k_P$ is given in
Tables \ref{TAB: type I short} and \ref{TAB: exceptional}.
\end{prop}
For the case of reducible symmetric spaces we point the reader to Proposition \ref{propprod}.

Theorem~\ref{Main_submanifolds_shape_Operator} presents the details of this approach, but before stating the theorem, let us explain two points.

First, we need to recall the ``Cartan type''. We say that two classical irreducible symmetric spaces of compact type
have the \emph{same Cartan type} if they have a common \emph{Cartan symbol} (CS)
as given in Table \ref{TAB: Cartan type}. For further details see Section \ref{secsym}.

Second, Theorem~\ref{Main_submanifolds_shape_Operator} works with different notions of equivalence. In order to avoid a repetition and trying to merge these notions into one statement,
  we first fix a notion of equivalence before stating the theorem. Accordingly, we state that ``Q is isomorphic to a symmetric space of compact type'' if it is either homotopy equivalent, homeomorphic, diffeomorphic, or isometric (by abuse of notation, of course, the latter notion does incorporate the possibility of applying different scaling factors on the different de Rham factors of a symmetric space) to a symmetric space of compact type. Having chosen such a notion of equivalence, this determines the term ``isomorphic'' throughout the theorem.

Note further that in the assertion of the theorem $\foc_{Q}$ is the focal radius of $Q$, $S_v$ is the shape operator of $Q$ corresponding to a unit vector $v$ normal to $Q$ (see Subsection~\ref{SS_Intermediate} for more details) and the quantities $k_{P}$ and $C_{P}$ are the respective ones from Table~\ref{TAB: type I short} for higher rank spaces and respectively $1$ and $(n-9)/2$ for the rank one spaces.

\begin{main}\label{Main_submanifolds_shape_Operator}
Let $Q$ be a compact connected embedded proper submanifold of an irreducible simply-connected compact classical symmetric space $P$ with $\Ric_{k_{P}}\geq \delta$, for some $\delta>0$. Assume further that $Q$ is isomorphic to a symmetric space of compact type and satisfies the following condition:
There exists some $r\in [0, \pi/2)$ with
$$\foc_{Q}>r$$
  such that for every $x\in Q$, every unit normal vector $v\in T_{x}Q^{\perp}$, and any $k_{P}$-dimensional subspace $W$ of $T_{x}Q$ we have that
$$\lvert\trace(\shape_{v}\mid_{W})\rvert\leq \sqrt[\leftroot{-2}\uproot{2}]{\frac{\delta}{k_{P}}}k_{P}\cot\bigg(\pi/2- \sqrt[\leftroot{-2}\uproot{2}]{\frac{\delta}{k_{P}}}r\bigg),$$
Then if $\Codim~Q\leq C_{P}$,
one of the following cases occurs:
\begin{itemize}
\item[1.]
If $P$ is isometric to a sphere, then $Q$ is isomorphic to a product of spheres whose dimensions are at least $10$.
\item[2.] If $P\cong \tfrac{\SO(2+q)}{\SO(2)\times \SO(q)}$, $q\geq 10 $, then $Q$ is isomorphic to
 a symmetric space $ \tfrac{\SO(2+q')}{\SO(2)\times \SO(q')}$, $q'<q$, possibly up to products with spheres of dimensions at least $10$, or $Q$ is isomorphic to a complex projective space $\mathbb{CP}^{n}$, $n\geq 5$, possibly up to products with spheres of dimensions at least $10$.

 Similarly, if $P^{n}$ is isometric to a complex projective space with $n\geq 11$, then $Q$ is isomorphic to a complex projective space $\mathbb{CP}^{r}$, $r\geq 5$, possibly up to products with spheres of dimensions at least $10$, or $Q$ is isomorphic to a Grassmannian manifold $ \tfrac{\SO(2+q)}{\SO(2)\times \SO(q)}$ with $q\geq 10$, possibly up to products with spheres of dimensions at least $10$.
 \item[3.] If $P$ is isometric to a Grassmannian manifold (other than those appearing in Items 1 and 2), then
  $Q$ is isomorphic to a symmetric space with the same Cartan type as $P$, possibly up to products with spheres of dimensions at least $10$.
 \item[4.] If $P$ is none of the above symmetric spaces, then $Q$ is isomorphic to a necessarily reducible symmetric space of the form $Q_{1}\times S^{l_{1}}\times \ldots \times S^{l_{r}}$, where $Q_{1}$ has the same Cartan type as $P$ and $l_{i}\geq 10$, for $1\leq i\leq r$.
\end{itemize}

\end{main}

\begin{rem}\label{R:Valid_dim}
Note that if $\dim P$ is not ``large'' enough, then $C_{P}< 1$. Thus there is no (proper) submanifold $Q$ of $P$ with $\codim~Q\leq C_{P}<1$. For example, for a compact rank one symmetric space $P$, if $\dim P\leq 10$, then $C_{P}<1$.
For this reason, we say that $\dim P$ is a \emph{valid} dimension if $C_{P}\geq 1$. For instance the valid dimension of a compact rank one symmetric space is $11$.

In Theorem~\ref{Main_submanifolds_shape_Operator}, we implicitly assume that $\dim~P$ is a valid dimension.
\end{rem}

Our approach for proving Theorem \ref{Main_submanifolds_shape_Operator}
comprises three major steps.
\begin{enumerate}
\item First we present a uniform method of how to compute the smallest $k$ for which the symmetric space $P$ has $k$-positive Ricci curvature. This will follow from a detailed analysis of associated root systems and isotropy orbits.
\item Next, we use this information as one key ingredient in the generalized connectedness lemma by Guijarro--Wilhelm. As an outcome in the respective cases we gain control on the degree of connectedness of the embedding, and we arrange codimensions in such a way that the map is $10$-connected.
\item In particular, this implies that the first $9$ homotopy groups of the ambient manifold and the submanifold have to coincide. We hence collect and compute nearly all such homotopy groups for \emph{all} irreducible simply-connected symmetric spaces, and compare them with the ones of the ambient space. This lets us conclude that Cartan types of ``almost all'' cases
 necessarily agree in the codimension ranges we consider.
 \end{enumerate}

This procedure illustrates and varies nicely---in this form in the very first time---the traditionally fruitful interplay of Riemannian geometry (here in the form of $k$-positive Ricci curvature), topology (in the form of connectedness and homotopy), and algebra (as homotopy groups, and used to compute $k$-positive Ricci) on symmetric spaces. We remark that, in particular, Points (1) and (3) may be of independent interest. For example, as a corollary of these observations in Point (3) we can fix
\begin{maincor}\label{cor01}
The homotopy groups up to degree $9$ of a (simply-connected) symmetric space of compact type uniquely identify the Cartan types except for the blind spot of not differentiating between $\SO(2+q)/\SO(2)\SO(q)$ and $\cc\pp^{n}$ for $q$ and $n$ as in Theorem~\ref{Main_submanifolds_shape_Operator}.

\end{maincor}
We remark that this improves a similar result for Lie groups (see \cite[Theorem 1]{Boe98}) by restricting to degrees at most $9$ and moreover generalizes the result to the much larger regimen of symmetric spaces.
In particular, note again, that the only information we draw from knowing the entire homotopy type in the theorem are just the first $9$ homotopy groups. Similar statements using rational homotopy groups, or combinations of rational and $\zz_2$-homotopy groups only do not hold.

\bigskip

As already announced this new generalized approach applying to more general submanifolds even has previously unknown consequences for totally geodesic submanifolds.

We denote by $\Gr(p,n)$ the simply-connected Grassmannian of oriented real $p$-planes in $\rr^n$ (here by a slight abuse of notation), or complex $p$-planes in $\cc^n$, or quaternionic ones in $\hh^n$, respectively. Furthermore, $C_P$ denotes the corresponding number from Table \ref{TAB: type I short}.

\begin{main}\label{Main_T_g_Range}
Let $P=\Gr(p,n)$ be as above
with  $3\leq p<n/2$, and let $Q$ be a complete totally geodesic embedded submanifold of $P$ which satisfies $\operatorname{ind} P\leq \codim Q\leq C_P$, where $\operatorname{ind} P$ is the index of $P$. Then $Q$ has the same Cartan type as $P$.
\end{main}

\begin{rem}
It is interesting to note that restricting to the category of totally geodesic submanifolds in most cases our approach allows us to reprove and to actually \emph{even improve} the classical estimate
\begin{align*}
\operatorname{ind} P \geq \rk P
\end{align*}
stating that the codimension of a totally geodesic submanifold of $P$ is at least as large as its rank. This was originally proved in \cite{BerndtOlomos_index} and was the starting point for a sequence of case by case considerations which finally led to a confirmation of the index conjecture (see \cite{BO_Conj}).

Unfortunately, a confirmation and reproof of the entire index conjecture a priori does not lie within the scope of our methods without further refinement or finetuned arguments tailored to the respective Cartan types---even in those situations in which our outcomes are essentially better than the rank estimate above. We can make a simple example to illustrate that. In the Lie group $\Sp(n+1)$ we have the (totally geodesic) Lie subgroup $\Sp(n)\times \Sp(1)$ which realises the index. However, its inclusion is just $3$-connected. This illustrates that the index (and the codimensions close to it) may be realized by submanifolds whose topology does not have strong enough ties to the one of the ambient space in order to control their Cartan type.
\end{rem}

\smallskip

\str In Section \ref{sec01} we recall the necessary ingredients from $k$-positive Ricci curvature and from symmetric spaces. Section \ref{sec02}
is devoted to computing the quantity $k_P$ in Proposition \ref{P:Ric_k}. There we set up a scheme in which these computations are obtained, and provide tables of the concrete outcome.
Section \ref{sec05} collects tables of all the homotopy groups of symmetric spaces up to degree $9$, and gives either precise references for them or shows how to compute them explicitly. All this finally culminates in the proofs of the main theorems in Section \ref{S:Proof of main Theorems}.

\smallskip

\ack The authors would like to thank Luis Guijarro and Fred Wilhelm for explaining to them some subtleties of the generalized connectedness lemma and to thank Makiko Sumi Tanaka for a clarification on the polar--meridian pairs. The authors are also grateful to J\"urgen Berndt for commenting on a previous version of the article.
%\textcolor{red}{
We thank
Miguel Domínguez-Vázquez, David Gonz\'alez-\'Alvaro and Lawrence Mouill\'e
 for pointing us to some typos, to a flaw in Formula \eqref{EQ: Max dim G2}
and for making us aware of the references \cite{Lee} and \cite{Liu}.%}

The first named author was supported both by a Heisenberg grant and his research grant AM 342/4-1 of the German Research Foundation.
The third named author was supported by the DFG grant AM 342/4-1; the first and third named authors are members of the DFG Priority Programme 2026.

%+++++++++++++++++++++++++++++++++++++++++++++++++++++++++++++++++++++++++++++++++++++++++
% Preliminaries
%+++++++++++++++++++++++++++++++++++++++++++++++++++++++++++++++++++++++++++++++++++++++++
\section{Preliminaries}\label{sec01}

%%%%%%%%%%%%%%%%%%%%%%%%%%%%%%%%%%%%%%%%%%%%%%%%%%%%%%%%%%%
\subsection{Intermediate Ricci curvature}\label{SS_Intermediate}

A Riemannian manifold has \emph{$k$-th intermediate Ricci curvature} greater than or equal to $l$, denoted by $\Ric_{k}\geq l$, if for any orthonormal $(k+1)$-frame $\{v,w_1,\ldots, w_k\}$ the sum of sectional curvatures satisfies
\begin{align*}
\sum_{i=1}^k \sec(v,w_i)\geq l.
\end{align*}
Note that it is sufficient to assume $v,w_1,\ldots, w_k$ only to be orthogonal here; normalization reconciles them with the context of Ricci curvature.

Let $N$ be a smoothly embedded submanifold of a Riemannian manifold $M$, and denote by $S_{v} \colon T_{p}N\to T_{p}N $ the shape operator of $N$ corresponding to a unit vector $v$
normal to $N$. For brevity we write $S_{v}\mid_{W}$ for the composition of $S_{v}$ restricted to some subspace $W $of $T_{p}N$ with orthogonal projection $T_{p}N \to W$. We also denote by $\foc_{N}$ the focal radius of N. Then

\begin{theo}[\protect{\cite[Theorem B]{WG19}}]\label{T_Connectivity_Principle}
Let $M$ be a simply-connected, complete Riemannian $n$–manifold with $\Ric_{k}\geq k$, and let $N \subseteq M$ be a compact, connected, embedded, $l$–dimensional submanifold. If for some $r\in [0, \pi/2)$,
$$\foc_{N}>r$$
 and for every $x\in N$, $v\in T_{x}N^{\perp}$, and any $k$-dimensional subspace $W$ of $T_{x}N$,
 $$\lvert\trace(\shape_{v}\mid_{W})\rvert\leq k\cot(\pi/2-r),$$
then the inclusion $N\hto {} M$ is $(2l-k-n+2)$-connected.
\end{theo}

From Theorem~\ref{T_Connectivity_Principle} we have the following corollary.

\begin{cor}%[connectivity principle]
\label{C:Connectedness_Principle}
Let $M$ be a simply-connected complete closed Riemannian $n$-manifold with $\Ric_k M > 0$ and let $N\In M$ be a compact connected $l$-dimensional totally geodesic submanifold. Then the inclusion $N\hto{} M$ is $(2l-n-k+2)$-connected.

\end{cor}
\begin{rem}
Corollary~\ref{C:Connectedness_Principle} follows from Theorem~\ref{T_Connectivity_Principle} by observing that
\begin{itemize}
\item Up to rescaling the condition $\Ric_kM\geq k$ can be guaranteed: by compactness there is a uniform positive lower bound on $\Ric_k M$, say $\delta$. If we rescale the metric by $\delta/k$, we get that $\Ric_kM\geq k$ with respect to this rescaled metric.
\item Compactness implies the existence of a uniform positive lower
bound on the focal radius; whence we can assume that $r=0$.
\item The trace estimate for the second fundamental form is trivial due to the totally geodesic embedding, which is preserved under rescaling.
\end{itemize}
\end{rem}

\begin{rem}
As we have seen, the conditions in Corollary \ref{C:Connectedness_Principle} are satisfied for totally geodesic submanifolds. They are, however, not necessarily satisfied for \emph{minimal} submanifolds, since vanishing of mean curvature does not imply the vanishing of its restrictions to all subspaces $W$. Indeed, there are various examples of minimal codimension $1$ spheres in compact symmetric spaces (see for example \cite{Hsiang1988}).
\end{rem}

\begin{prop}[\protect{\cite[Theorem 1.1]{GWFocal}}]\label{P:Focal_Radius}
Let $M$ be a complete Riemannian $n$-manifold with $\Ric_{k}\geq k$ and $N$ be any
submanifold of $M$ with $\dim(N)\geq k$. Then $\foc_{N}\leq \pi/2$.
\end{prop}

%%%%%%%%%%%%%%%%%%%%%%%%%%%%%%%%%%%%%%%%%%%%%%%%%%%%%%%%%%%

\subsection{Symmetric Spaces}\label{secsym}
 We first recall and introduce the notions from the theory of symmetric spaces
needed in this paper.
Further details can be found in the standard literature e.g.\ \cite{He} and \cite{L-II}.\par

Let $G$ be a semi-simple compact connected Lie group.
A pair $(G,K)$ is called a \emph{symmetric pair} if $K$ is an open and closed
subgroup of the fixed point set $G^\sigma$ of an involutive automorphism $\sigma$ of $G.$
A symmetric pair $(G,K)$ induces
a splitting
$$\lieG=\lieK\oplus \lieP$$
of the Lie algebra $\lieG$ of $G,$
where $\lieP$ is the $(-1)$-eigenspace of the differential $\sigma_*$ of
$\sigma$ at the identity. The Lie algebra $\lieK$ of $K$ is
the set of fixed points of $\sigma_*.$\par
A symmetric pair $(G,K)$
is \emph{associated with} a connected Riemannian symmetric space $P$ of compact type if
$G$ acts transitively and by
isometries on $P$ such that the kernel of this action is finite, and such that $K$ is
the isotropy group of $G$ at a base point $o\in P$. In this case we also say that the pair
$(\lieG,\lieK)$ consisting of the Lie algebras of $G$ and $K$ is \emph{associated with} $P$.

\bigskip

The compact Lie groups come in certain series---like $\SU(n)$, $\Spin(n)$, $\Sp(n)$, etc. As observed by Cartan, so do the compact symmetric spaces, i.e.~they fall into few different types, the so-called \emph{Cartan types}. Already in the introduction we specified that two symmetric spaces share a Cartan type if they have a common Cartan symbol. In Table \ref{TAB: Cartan type} we present the list of such symbols.
\begin{table}[ht]\centering
\caption{Cartan symbols of classical irreducible symmetric spaces of compact type
described by infinitesimal symmetric pairs $(\lieG,\lieK)$, see \cite[Chap.\ X, \S 6]{He}}
\label{TAB: Cartan type}
\begin{tabular}{ccccccccc}
%\toprule
CS&$\lieG$ & $\lieK$ & Conditions \\
\midrule
$\mathsf{A}$ & $\lieSU_{n}\times\lieSU_{n}$ & $\Delta\lieSU_{n} $& $n\geq 2$\\
%\hdashline
$\mathsf{A\, I}$& $\lieSU_n$ & $\lieO_n$ & $n\geq 2$ \\
%\hdashline
$\mathsf{A\, II}$ & $\lieSU_{2n}$ & $\lieSp_n$ & $n\geq 2$ \\
%\hdashline
$\mathsf{A\, III}$
& $\lieSU_{p+q}$ & $\mathfrak{s}(\lieU_p\oplus\lieU_q)$ & $1\leq p\leq q$ \\
%\hdashline
$\mathsf{BD}$&$\lieO_{n}\times \lieO_{n}$& $\Delta\lieO_n$ & $n\geq 3,\; n\neq 4$ \\
%\hdashline
$\mathsf{BD\, I}$
 & $\lieO_{p+q}$ & $\lieO_p\oplus \lieO_q$ & $1\leq p\leq q,\; (p,q)\notin \{(1,1), \, (2,2)\}$ \\
%\hdashline
$\mathsf{C}$& $\lieSp_n\times \lieSp_n$& $\Delta\lieSp_n$ &$n\geq 1$ \\
%\hdashline
$\mathsf{C\, I}$& $\lieSp_{n}$ & $\lieU_n$ & $n\geq 1$\\
%\hdashline
$\mathsf{C\, II}$& $\lieSp_{p+q}$ & $\lieSp_p\oplus\lieSp_q$ & $1\leq p\leq q$ \\
%\hdashline
$\mathsf{D\, III}$& $\lieO_{4n}$ & $\lieU_{2n}$ & $n\geq 2$\\
%\bottomrule
\end{tabular}
\end{table}
Note that some low dimensional symmetric
spaces may have several Cartan symbols due to isomorphisms which are explained in \cite[Chap.\ X, \S 6]{He} (see Table~\ref{mastertable_SI} for a list of such isomorphisms). Observe further that all compact simple rank-one Lie groups have the same Cartan type.

\bigskip

The natural projection
$$\pi: G\to P,\qquad g\mapsto g.o$$
induces an identification of the quotient $G/K$ with $P$
such that the involution $\sigma,$ which descends to $G/K,$
becomes the geodesic symmetry of $P.$
The differential $\pi_*$ of $\pi$ at the identity induces a linear isomorphism
$$\pi_*|_{\lieP}:\lieP\to T_oP.$$
Using this identification the Riemannian metric on $T_oP$ is induced by a suitable bi-invariant
metric $\langle\, ,\, \rangle$ on $\lieG$ restricted to $\lieP$, and the sectional curvature of a 2-plane
 spanned by two perpendicular unit vectors $v,w\in \lieP$ is given by
\begin{equation}
\label{EQ: Sec}
 \sec(v,w)=\big\langle [v,w],[v,w]\big\rangle=\big\|[v,w]\big\|^2.
\end{equation}

The rank $r$ of a symmetric space $P=G/K$ is the dimension of any maximal abelian subspace of
$\lieP.$ We now fix a maximal abelian subspace $\lieA$ of $\lieP$ and denote by
$\lieA^*$ its dual vector space. For each $\alpha\in\lieA^*$ we set
$$\lieG_\alpha:=\{Y\in\lieG\otimes \C\; |\; [H,Y]=\sqrt{-1}\alpha(H)Y\; \text{for all}\;
H\in\lieA\}.$$
Then the root system of $P$ is
$$\Root:=\{\alpha\in\lieA^*\; |\; \alpha\neq 0\; \text{and}\; \lieG_{\alpha}\neq\{0\}\}.$$
One can now choose a basis
$$\Sigma=\{\alpha_1,\dots,\alpha_r\}\subset\Root$$
of $\lieA^*,$ called system of \emph{simple roots}, which enjoys the property that
any root $\alpha\in\Root$ is a linear combination of the elements of $\Sigma$ with
either only non-positive or only non-negative integer coefficients.
Let $$\{\xi_1,\dots,\xi_r\}\subset\lieA$$ be the dual basis of $\Sigma,$ that is
$\alpha_j(\xi_k)=0$ if $j\neq k$ and $\alpha_j(\xi_j)=1.$
The elements of the set
$$\Root^+:=\left\{\alpha\in\Root\; |\; \forall j\in\{1,\dots, r\}: \;
\alpha(\xi_j)\geq 0\right\}$$ are called
\emph{positive roots} w.r.t.\ $\Sigma.$ Further, for each $\alpha\in \Root^+$,
$$
 \alpha=\sum\limits_{j=1}^rc^\alpha_j\alpha_j \qquad\text{with}\; c^\alpha_j =\alpha(\xi_j)\in\N\cup\{0\}\;
 \text{for all}\; j\in \{1,\dots ,r\}.
$$
We have the following direct sum decomposition:
\begin{equation}
 \label{EQ: direct sum deco}
  \lieP=\lieA\oplus\sum\limits_{\alpha\in\Root^+}\lieP_\alpha \quad \text{with}
 \quad \lieP_\alpha=(\lieG_\alpha\oplus\lieG_{-\alpha})\cap\lieP.
\end{equation}

For $\alpha\in\Root^+$ we call $m_\alpha:=\mathrm{dim}(\lieP_\alpha)$
the \emph{multiplicity} of $\alpha$.\par

\subsection{Polar-Meridian pairs} In this subsection we quickly review the notions of polars and meridians of compact symmetric spaces---introduced by Chen--Nagano---and recall some important facts about them which are relevant to our article. We refer the reader to \cite{ChenNagano, Nagano, NaganoII} for further information.

Let $P$ be a compact symmetric space, and $o$ be the base point.
Let $p\neq o$---as the midpoint of a closed
geodesic emanating from $o$---be a fixed point of $s_o$, the geodesic symmetry of $P$ at $o$.
The connected component of the fixed point set
of $s_o$ containing $p$ is called a \emph{polar} of $o$ at $p $ and is denoted by $P^{o}_{+}(p)$ or simply by $P_{+}(p)$ or $P_{+}$ if less/no precision is needed. If $P^{o}_{+}(p)=\{p\}$, then $p$ is called a \emph{pole} of $o$. We call the connected component of the fixed point set of $s_p\circ s_o$
containing $p$ the \emph{meridian} to $P^{o}_{+}(p)$ and denote it by $P^{o}_{-}(p)$ or simply by $P_{-}(p)$ or $P_{-}$. Clearly, both $P_{+}(p)$ and $P_{-}(p)$ are complete totally geodesic submanifolds of $P$, and hence they are compact symmetric spaces as well. Moreover, the tangent space of $P_{-}(p)$ at $p$ is the orthogonal complement of the tangent space $T_{p}P_{+}(p)$ in $T_{p}P$.

Let $p$ be an antipodal point of $o$. Then we have a quadruple \linebreak
 $ (o, p, P^{o}_{+}(p), P^{o}_{-}(p)) $ as above. Let $G_{P}$ denote the \emph{symmetry group}, i.e. the closure of the group of isometries generated by
$\{s_{q} \mid q\in M\}$ in the compact-open topology, where $s_{q}$ is the geodesic symmetry of $P$ at $q$. We say that two quadruples $(o, p, P^{o}_{+}(p), P^{o}_{-}(p))$ and $(o', p', P^{o'}_{+}(p'), P^{o'}_{-}(p'))$ are equivalent if there is an isometry $\varphi\in G_{P}$ which carries
$(o, p, P^{o}_{+}(p), P^{o}_{-}(p))$ to $(o', p', P^{o'}_{+}(p'), P^{o'}_{-}(p'))$ respectively.
We denote the set of all equivalence classes of such quadruples by $\mathcal{P}(P)$. Note that $\mathcal{P}(P)$ is a finite set and globally determines $P$ in the sense that two (compact and connected) symmetric spaces $P_{1}$ and $P_{2}$ are isomorphic if and only if $\mathcal{P}(P_{1})$ is isomorphic to $\mathcal{P}(P_{2})$ (see e.g.~\cite[Theorem~5.1]{ChenNagano}).

Let $Q$ be a complete totally geodesic submanifold of $P$, which is in turn a compact symmetric space. Now we briefly explain how the polar-meridian pairs of $P$ and those of $Q$ are related. Any isometric totally geodesic embedding $f\colon Q\to P $ gives rise to a mapping $\mathcal{P}(f)\colon \mathcal{P}(Q)\to \mathcal{P}(P)$
sending $(o, p, Q^{o}_{+}(p), Q^{o}_{-}(p)) $ into $(f(o), f(p), P^{f(o)}_{+}(f(p)), P^{f(o)}_{-}(f(p))$. Note that $\mathcal{P}(f)$ is well-defined, since every element $\varphi$ of $G_{Q}$ is the restriction of an element $\psi$ of $G_{P}$
such that we have $f\circ \varphi=\psi\circ f$. Moreover, we recall

From \cite[Theorem~2.9., and the paragraph preceding the theorem]{ChenNagano} we cite
\begin{lemma}\label{L_Inclusion_Pairs} %[\protect{\cite[Theorem~2.9., and the paragraph preceding the theorem]{ChenNagano}}]
\begin{itemize}
\item[1.] $f(Q^{o}_{+}(p))\subseteq P^{f(o)}_{+}(f(p)),$ and $f(Q^{o}_{-}(p))\subseteq P^{f(o)}_{-}(f(p))$.
\item[2.] $f$ induces a pairwise totally geodesic immersion $\mathcal{P}(f)\colon \mathcal{P}(Q)\to \mathcal{P}(P) $.
\end{itemize}
\end{lemma}

%*************************************************************************************************************************************
% The determination of $k_P.
%*************************************************************************************************************************************

\section{The determination of $k_P.$}\label{sec02}

The main objective of this section is to prove Proposition \ref{P:Ric_k}. We first show that for
every irreducible simply-connected compact symmetric space $P$ there exists an integer $k_{p}$
such that $\Ric_{k_{P}}>0$ and then compute $k_{p}$ for each such symmetric space $P$.

%\textcolor{red}{
Our calculations are pretty straightforward. The method we adopt
is already oulined in \cite[pp.\ 110 -- 111]{Lee} in a similar context, but has then been abandoned in \cite{Lee} in
favor of (rather lengthy) elementary calculations which again were merely applied to some classical cases. A slightly different approach is used in \cite{Liu}, but the case of simple compact Lie groups seems to be still missing. After the first version of our paper had appeared on the arXiv in October 2020,
we were informed that Dom\'inguez-V\'azquez, Gonz\'alez-\'Alvaro, and Mouill\'e were in the process of independently determining $k_P$ by a different method
(see \cite{DV-GA-M}).%}

%\textcolor{red}{
Before we start the outlined reasoning%}
, let us quickly remark on the fact that it suffices to establish this for irreducible symmetric spaces.
\begin{prop}\label{propprod}
Let $(M_1\times M_2, g_1\oplus g_2)$ be the Riemannian product of two closed Riemannian manifolds $(M_1^n,g_1)$ and $(M_2^m,g_2)$. Suppose that at two respective points $\Ric_{k_1} M_1>0$ and $\Ric_{k_2} M_2>0$. Then, at the product base point, $\Ric_k (M_1\times M_2, g_1\oplus g_2)>0$ for
\begin{align*}
k=\max \{n + k_2, m+k_1\}
\end{align*}
\end{prop}
\begin{proof}
Fix an orthonormal frame $x_1,\ldots, x_n$ in $T_pM_1$ and one, $y_1, \ldots, y_m$, in $T_qM_2$. Due to the product metric we obtain that $\sec(x_i,y_j)=0$ for all $i,j$; and so it may become necessary to sum up over $n+k_2$ respectively $m+k_1$ many directions in order to obtain positive intermediate Ricci.
\end{proof}
Clearly, on compact symmetric spaces due to homogeneity we do not need to worry about base points, and, due to non-negative curvature, we obtain positive intermediate Ricci for all $n+m> k\geq \max \{n + k_2, m+k_1\}$.

Since a compact symmetric space $P$ has non-negative sectional curvature,
it has also non-negative $k$-th intermediate Ricci curvature for any $k\in\N.$
Equation \eqref{EQ: Sec} now immediately implies that
$P$ has positive $k$-th intermediate Ricci curvature,
if and only if
$$k\geq k_P:=\max\limits_{\xi\in\lieP\setminus\{0\}}\dim(C_{\lieP}(\xi))$$
where $C_{\lieP}(\xi)=\{v\in\lieP \mid [v,\xi]=0\}$ is the centralizer of $\xi$ in $\lieP.$
This already shows the existence part in Proposition \ref{P:Ric_k}.
Using the decomposition given in \eqref{EQ: direct sum deco} we see that
$$C_{\lieP}(\xi)=\lieA \oplus\sum\limits_{\substack{\alpha\in\Root^+\\ \alpha(\xi)=0}}\lieP_\alpha.$$
Note that $C_{\lieP}(\xi)$ is the normal space at $\xi$ of the linear isotropy orbit of
$\xi$ (see e.g.\ \cite[pp.\ 49--51]{BCO}). Thus $k_P$ is
the maximal codimension of a non-trivial orbit of the linear isotropy representation of $P$.\par
Since
\begin{equation}
 \label{EQ: Max dim1}
 \begin{array}{rl}
k_P
 &=\dim(\lieA) +\max\limits_{\xi\in\lieP\setminus\{0\}}
 \Big(\sum\limits_{\substack{\alpha\in\Root^+\\ \alpha(\xi)=0}}\lieP_\alpha\Big)
  = r +\max\limits_{\xi\in\lieP\setminus\{0\}}
 \Big(\sum\limits_{\substack{\alpha\in\Root^+\\ \alpha(\xi)=0}}m_\alpha\Big)\\
&=r+\max\limits_{j\in\{1,\dots, r\}}\Big(\sum\limits_{\substack{\alpha\in\Root^+\\ \alpha(\xi_j)=0}}
m_\alpha\Big)
 = r+\max\limits_{j\in\{1,\dots, r\}}\Big(\sum\limits_{\alpha\in\Root_j^0}m_\alpha\Big),
\end{array}\end{equation}
where
$$\Root_j^0=\{\alpha\in\Root^+:\; \alpha(\xi_j)=0\}=
\{\alpha\in\Root^+:\;c^\alpha_j=0\},$$
we see that $k_P$ is realised by the dimension of the centralizer within $\lieP$ of a dual vector
of a simple root (case-by-case inspections below show that
these simple roots are end-nodes of the Dynkin diagram of $P$).
This reduces the calculation of $k_P$
to considerations of the root system with multiplicities of $P$.

We denote by $d_P$ the minimal dimension of a non-trivial linear isotropy orbit of $P$,
i.e.~$$d_{P}:=\dim(P)-k_P.$$

For later use we define the \emph{connectivity number} of a totally geodesic submanifold $Q$ of $P$ as
$$\sharp_P(Q)=2+d_{P}-2\cdot\Codim(Q).$$

\subsection{Case-by-case determination of $k_P$}\label{sec_cases}
We now assume that $P$ is an irreducible symmetric space of compact
type. Then its root system $\Root$ is irreducible.
The necessary details about the different reduced root systems can be found in the appendix of Bourbaki's book
\cite[Planches I--IX]{Bourbaki}. The positive
roots of a root system of exceptional type $E$ and $F$ are nicely listed in the Hasse diagrams
in the corresponding Wikipedia articles (retrieved October
27th, 2020)

\url{https://en.wikipedia.org/wiki/E6_(mathematics)},

\url{https://en.wikipedia.org/wiki/E7_(mathematics)},

\url{https://en.wikipedia.org/wiki/E8_(mathematics)}, and

\url{https://en.wikipedia.org/wiki/F4_(mathematics)}.

\noindent The information on the non-reduced root systems
is taken from \cite[Thm.\ 3.25, p.\ 475]{He}.

In all the cases we provide the formulas for $k_P$.
Via the information from Table \ref{TAB: type I Helgason list}
on the multiplicities of the short, long and extra long roots we then,
in a second step, provide the explicit values in Table \ref{TAB: type I short}.

\subsubsection{Rank one symmetric spaces}
Since the linear isotropy action of rank one symmetric spaces is transitive on the unit sphere,
we have
$$k_p=1.$$

\subsubsection{Simply laced classical reduced root systems}
An irreducible root system is called \emph{simply laced}, if all roots have the same length. In this case
the Weyl group of $P$ acts transitively on $\Root$ and all roots therefore have equal multiplicity $m.$
By Equation \eqref{EQ: Max dim1} we get:
\begin{equation}
 \label{EQ: Max dim simply laced}
k_{P}
=r+m\cdot \max\limits_{j\in\{1,\dots, r\}}|\Root^0_j|,
\end{equation}
where $|\Root^0_j|$ denotes the cardinality of $\Root^0_j.$

\begin{enumerate}[(i)]
\item \emph{Type $A_r,\; r\geq 2$}.\\
We represent $A_r$ by
$\Root^+=\{e_j-e_k:\; 1\leq j<k\leq r+1\}$ and
$\Sigma=\{\alpha_1=e_1-e_2, \, \alpha_2=e_2-e_3,\,\dots,\, \alpha_{r}=e_{r}-e_{r+1}\}$
as a system of simple roots.
With $|\Root^+|=\frac{1}{2} r(r+1)$ and
$$e_j-e_k=\sum\limits_{l=j}^{k-1} \alpha_l\quad\text{for}\quad 1\leq j<k\leq r+1$$
we get $\max\limits_{j\in\{1,\dots, r\}}|\Root^0_j|=|\Root^0_1|=\frac{1}{2}r(r+1)-r=\frac{1}{2}r(r-1)$ and therefore,
\begin{equation}
 \label{EQ: Max dim A}
k_{P}
=r+\frac{1}{2}r(r-1)m.
\end{equation}
\item \emph{Type $D_r$.}\\
Note that $D_3\cong A_3$ and $D_2\cong A_1\times A_1$
is not irreducible. For $r\geq 4$ we represent $D_r$ by
$\Root^+=\{e_j\pm e_k:\; 1\leq j<k\leq r\}$ with
$\Sigma=\{\alpha_1=e_1-e_2, \, \dots,\, \alpha_{r-1}=e_{r-1}-e_r,\, \alpha_r=e_{r-1}+e_r\}$
as system of simple roots.
We have $|\Root^+|=r(r-1)$ and
\begin{itemize}
 \item $e_j-e_k=\sum\limits_{l=j}^{k-1} \alpha_l$ for $1\leq j<k\leq r.$
 \item $e_j+e_k=\sum\limits_{l=j}^{k-1} \alpha_l+2\sum\limits_{l=k}^{r-2} \alpha_l
 +\alpha_{r-1}+\alpha_r$ for $1\leq j<k\leq r-1.$
  \item $e_j+e_r=\sum\limits_{l=j}^{r-2} \alpha_l+\alpha_r$ for $1\leq j\leq r-1.$
  \end{itemize}
We now see that $\max\limits_{j\in\{1,\dots, r\}}|\Root^0_j|=|\Root^0_1|=r(r-1)-2(r-1)=(r-1)(r-2)$
and therefore,
\begin{equation}
 \label{EQ: Max dim D}
k_{P}
=r+m(r-1)(r-2)\quad \text{for}\; r\geq 4.
\end{equation}
\item \emph{Type $E_r,\; r\in\{6,7,8\}$}.\\
Looking at the Hasse diagrams in the aforementioned Wikipedia articles or at \cite[Planches V--VII]{Bourbaki}
one sees that $\max\limits_{j\in\{1,\dots, r\}}|\Root^0_j|$ is always realized by
the simple root represented by the (right hand side) end-node of the longest branch
of the Dynkin diagram
\dynkin E6, \dynkin E7 or \dynkin E8 of $E_6,\; E_7$ or $E_8$, respectively.
\begin{center}
\begin{tabular}{c|ccc}
&$E_6$&$E_7$&$E_8$\\
\hline
$\max\limits_{j\in\{1,\dots, r\}}|\Root^0_j|$& 20 & 36 & 63\\
\end{tabular}
\end{center}
\end{enumerate}

\begin{equation}
k_{P}
=r+m\cdot \begin{cases}20 \textrm{\qquad for $E_6$}\\36 \textrm{\qquad for $E_7$}\\63 \textrm{\qquad for $E_8$}\\
\end{cases}
\end{equation}

\subsubsection{Non simply laced classical reduced root systems}
If $\Root$ is irreducible, non simply laced and reduced, then $\Root$ contains roots of exactly
two different lengths,
called \emph{short} and
\emph{long} roots.
Since the set of short roots $\Root^s$ and the set of long roots $\Root^l$
are orbits of the Weyl group of $P,$ all short roots
have equal multiplicity $m_s$ and all long roots have equal multiplicity $m_l$.
By Equation \eqref{EQ: Max dim1} we get
\begin{equation}
 \label{EQ: Max dim non simply laced}
k_{P}
=r+\max\limits_{j\in\{1,\dots, r\}}\left(m_s\cdot |\Root^s_j|
+m_l\cdot |\Root^l_j|\right),
\end{equation}

where $|\Root^s_j|$ and $|\Root^l_j|$ denote the cardinality of $\Root^s_j=\Root^s\cap\Root^0_j$
and $\Root^l_j=\Root^l\cap\Root^0_j,$ respectively.
\begin{enumerate}[(i)]
 \item \emph{Type $B_r$.}\\
We represent the root system of type $B_r$ by
$\Root^s\cap\Root^+=\{e_j: 1\leq j\leq r\}$
 and $\Root^l\cap\Root^+=\{e_j\pm e_k: 1\leq j<k\leq r\}$ with simple roots
$\Sigma=\{\alpha_1=e_1-e_2, \, \dots,\, \alpha_{r-1}=e_{r-1}-e_r,\, \alpha_r=e_r\}.$
We have $|\Root^+|=r^2$ and
\begin{itemize}
 \item $e_j=\sum\limits_{l=j}^r \alpha_l$ \, for \,$1\leq j\leq r.$
 \item $e_j-e_k=\sum\limits_{l=j}^{k-1} \alpha_l$ \, for \, $1\leq j<k\leq r.$
 \item $e_j+e_k=\sum\limits_{l=j}^{k-1} \alpha_l+2\sum\limits_{l=k}^{r} \alpha_l$
 for $1\leq j<k\leq r.$
\end{itemize}
We see that for $r\geq 4$ we have
 $$\max\limits_{j\in\{1,\dots, r\}}|\Root^s_j|=|\Root^s_1|=r-1$$
 and
  $$\max\limits_{j\in\{1,\dots, r\}}|\Root^l_j|=|\Root^l_1|=(r^2-r)-2(r-1)=(r-1)(r-2)$$
  and therefore,
\begin{equation}
 \label{EQ: Max dim B}
k_{P}
=r+m_s(r-1)+m_l(r-1)(r-2)\quad \text{for}\; r\geq 4.
\end{equation}
For $B_2$ and $B_3$ we get:
\begin{center}
\begin{tabular}{c p{15mm} c}
 \begin{tabular}{c|cc}
$j$ &$|\Root^s_j|$&$|\Root^l_j|$\\
\hline
1 & 1 & 0\\
2 & 0 & 1
\end{tabular}
 & &
 \begin{tabular}{c|cc}
$j$ &$|\Root^s_j|$&$|\Root^l_j|$\\
\hline
1 & 2 & 2 \\
2 & 1 & 1\\
3 & 0 & 3
\end{tabular}\\
\end{tabular}
\end{center}

\item\emph{Type $C_r$.}\\
We represent a root system of type $C_r$ by
$\Root^l\cap\Root^+=\{2e_j: 1\leq j\leq r\}$ and
$\Root^s\cap\Root^+=\{e_j\pm e_k: 1\leq j<k\leq r\}$ with simple roots
$\Sigma=\{\alpha_1=e_1-e_2, \, \dots,\, \alpha_{r-1}=e_{r-1}-e_r,\, \alpha_r=2e_r\}.$
We have $|\Root^+|=r^2$ and
\begin{itemize}
 \item $2e_j=2\sum\limits_{k=j}^{r-1} \alpha_k+\alpha_r$
 for $1\leq j\leq r.$
 \item $e_j-e_k=\sum\limits_{l=j}^{k-1} \alpha_l$ for $1\leq j<k\leq r.$
 \item $e_j+e_k=\sum\limits_{l=j}^{k-1} \alpha_l+2\sum\limits_{l=k}^{r-1} \alpha_l+\alpha_r$ for $1\leq j<k\leq r.$
\end{itemize}
Similar to the case $B_r$ we have that
$$\max\limits_{j\in\{1,\dots, r\}}|\Root^s_j|=|\Root^s_1|=(r^2-r)-2(r-1)=(r-1)(r-2)$$
and
$$\max\limits_{j\in\{1,\dots, r\}}|\Root^l_j|=|\Root^l_1|=r-1$$
  for $r\geq 4.$ Therefore
\begin{equation}
 \label{EQ: Max dim C}
k_{P}
=r+m_s(r-1)(r-2)+m_l(r-1)\quad \text{for}\; r\geq 4.
\end{equation}
Note that $C_2\cong B_2.$ For $C_3$ we get:
\begin{center}
\begin{tabular}{c|ccc}
$j$ &$|\Root^s_j|$&$|\Root^l_j|$\\
\hline
1 & 2 & 2& \\
2 & 1 & 1\\
3 & 3 & 0
\end{tabular}
\end{center}
\item \emph{Type $F_4$.} \\
Labelling the simple roots in the Dynkin diagram \dynkin F4 from left to right by $\alpha_1,\dots, \alpha_4$
and
looking at the Hasse diagrams in the aforementioned Wikipedia article or at \cite[Planches VIII]{Bourbaki}
we get
\begin{center}
\begin{tabular}{c|ccc}
$j$ &$|\Root^s_j|$&$|\Root^l_j|$\\
\hline
1 & 6& 3& \\
2 & 3 & 1\\
3& 1 & 3\\
4& 3 & 6
\end{tabular}
\end{center}\par
From the table in \cite[p.\ 534]{He} we see that for every irreducible symmetric space $P$ of type I whose root system is of type
$F_4,$ we have $m_l=1$ and $m_s\in\{1,2,4,8\},$ in particular $m_l<m_s$.
 We therefore get
\begin{equation}
 \label{EQ: Max dim F4}
k_{P}
=7+6 m_s,\qquad \text{if $P$ is of type I},
\end{equation}
which is realized by the orbit $X_{\{1\}}=\Ad_G(K_0)\xi_1.$

\item \emph{Type $G_2$}.\\
Since $G_2$ is reduced and has only two simple roots $\alpha_1$ and $\alpha_2$,
every non-simple positive root is of the form
$c_1\alpha_1+c_2\alpha_2$ with strictly positive integers $c_1$ and $c_2.$ Thus
$\max\limits_{j\in\{1,2\}}|\Root^s_j|=\max\limits_{j\in\{1,2\}}|\Root^l_j|=1$ and
\begin{equation}
 \label{EQ: Max dim G2}
k_{P}
=2+\max(m_s,m_l).
\end{equation}

\end{enumerate}

\subsubsection{Irreducible non-reduced root systems, type $BC_r$}
If $\Root$ is irreducible but non-reduced, then $\Root$ contains roots of exactly three different lengths,
called \emph{short}, \emph{long} and
\emph{extra long} roots.
The set $\Root^s\cap\Root^+=\{e_j: 1\leq j\leq r\}$ of positive short roots,
the set $\Root^l\cap\Root^+=\{e_j\pm e_k: 1\leq j<k\leq r\}$ of positive long roots
and the set $\Root^{xl}\cap\Root^+=\{2e_j: 1\leq j\leq r\}=2(\Root^s\cap\Root^+)$
of extra long roots are Weyl orbits. Therefore all short roots
have equal multiplicity $m_s,$ all long roots have equal multiplicity $m_l$ and all
extra long roots equal multiplicity $m_{xl}.$
By Equation \eqref{EQ: Max dim1} we get
\begin{equation}
 \label{EQ: Max dim non reduced}
k_{P}
=r+\max\limits_{j\in\{1,\dots, r\}}\left(m_s\cdot |\Root^s_j|+m_l\cdot|\Root^l_j|
+m_{xl}\cdot |\Root^{xl}_j|\right),
\end{equation}

where $|\Root^s_j|,\, |\Root^l_j|$ and $|\Root^{xl}_j|$ denote the cardinality
of $\Root^s_j=\Root^s\cap\Root^+\cap\Root^0_j, \;
\Root^l_j=\Root^l\cap\Root^+\cap\Root^0_j$ and
$\Root^{xl}_j=\Root^{xl}\cap\Root^+\cap\Root^0_j,$
respectively.\par
The simple root system corresponding to $\Root^+$ is (as for $B_r$)
$\Sigma=\{\alpha_1=e_1-e_2, \, \dots,\, \alpha_{r-1}=e_{r-1}-e_r,\, \alpha_r=e_r\}$
and $|\Root^+|=r^2+r.$
With
\begin{itemize}
 \item $e_j=\sum\limits_{l=j}^r \alpha_l$ \, for \, $1\leq j\leq r.$
 \item $e_j-e_k=\sum\limits_{l=j}^{k-1} \alpha_l$\, for \, $1\leq j<k\leq r.$
 \item $e_j+e_k=\sum\limits_{l=j}^{k-1} \alpha_l+2\sum\limits_{l=k}^{r} \alpha_l$\,
 for \, $1\leq j<k\leq r.$
  \item $2e_j=2\sum\limits_{l=j}^r \alpha_l$\, for \, $1\leq j\leq r.$
\end{itemize}
we observe that for $r\geq 4$
the maximal cardinalities of $\Root^s_j,\; \Root^l_j$ and $\Root^{xl}_j$ are all attained for $j=1.$
Similar to the $B_r$ case $B_r$ we have that
 $$\max\limits_{j\in\{1,\dots, r\}}|\Root^s_j|=|\Root^s_1|=r-1$$ and
  $$\max\limits_{j\in\{1,\dots, r\}}|\Root^l_j|=|\Root^l_1|=(r^2-r)-2(r-1)=(r-1)(r-2)$$
  and $$\max\limits_{j\in\{1,\dots, r\}}|\Root^{xl}_j|=|\Root^{xl}_1|=r-1$$ for $r\geq 4.$
  Therefore,
\begin{equation}
 \label{EQ: Max dim BC}
k_{P}
=r+(m_s+m_{xl})(r-1)+m_l(r-1)(r-2)\quad \text{for}\; r\geq 4.
\end{equation}
For type $BC_2$ and $BC_3$ we get:

\begin{center}
\begin{tabular}{c p{15mm} c}
 \begin{tabular}{c|ccc}
$j$ &$|\Root^s_j|$&$|\Root^l_j|$&$|\Root^{xl}_j|$\\
\hline
1 & 1 & 0 & 1\\
2 & 0 & 1 & 0
\end{tabular}
 & &
 \begin{tabular}{c|cccc}
$j$ &$|\Root^s_j|$&$|\Root^l_j|$&$|\Root^{xl}_j|$\\
\hline
1 & 2 & 2& 2\\
2 & 1 & 1& 1\\
3 & 0 & 3& 0
\end{tabular}\\
\end{tabular}
\end{center}

\subsection{Summarizing the results}

 In this subsection we first summarize the information on the multiplicities of short, long and extra long roots of irreducible simply-connected classical symmetric spaces of compact type in Table~\ref{TAB: type I Helgason list}, and then based on the computations in Subsection~\ref{sec_cases} we give explicit formulae for $d_{P}$, $k_{P}$, $C_{P}$ and $\sharp(Q)$ in Table \ref{TAB: type I short}. The counterpart information for the irreducible simply-connected exceptional symmetric spaces of compact type is provided in Table \ref{TAB: exceptional}. Note that the term $C_P$ relates to the other numbers via demanding connectivity to be at least $\#(Q)\geq 10$ and via
 \begin{align*}
 d_P- 2\cdot \codim Q+2\geq 10 \textrm{ with } C_P=d_P/2-4 \textrm{ such that }\codim Q\leq C_P.
 \end{align*}

We stress that the computation of the connectivity number $\sharp(Q)$ is essential to the forthcoming arguments: we will proceed by using this information in order to compare the low-degree homotopy groups---up to degree $10$---of a symmetric space and its potential submanifolds in the next section. This gives us a criterion to decide wether such a submanifold can occur or not. Note further that considering degrees of homotopy groups up to $10$ is enough to discern the different symmetric spaces.

%%%%%%%%%%%%%%%%%%%%%%%%%%%%%%%%%%%%%%%%%%%%%%%%%%%%%%%%%%%%. Tables with Cartan Symbol
%%%%%%%%%%%%%%%%%%%%%%%%%%%%%%%%%%%%%%%%%%%%%%%%%%%%%%%%%%%

\begin{table}[h]\centering
\caption{Structural information on classical irreducible symmetric spaces of compact type (taken from
\cite[Ch.\ X, \S 6]{He})}
\label{TAB: type I Helgason list}
\Rotatebox{90}{
\small
\begin{tabular}{ccccccccc}
%\toprule
CS&$\lieG$ & $\lieK$ &Cond.\ &Dynkin &$\dim(P)$ & $m_s$ & $m_l$ & $m_{xl}$ \\
\midrule
$\mathsf{A}$& $\lieSU_n\times \lieSU_n$ & $\Delta\lieSU_n$ & $2\leq n$ & $A_{n-1}$& $(n-1)(n+1)$ & & $2$ & \\
%\hline
$\mathsf{A\, I}$& $\lieSU_n$ & $\lieO_n$ & $2\leq n$ & $A_{n-1}$& $\frac{1}{2}(n-1)(n+2)$ & & $1$ & \\
%\hline
$\mathsf{A\, II}$ & $\lieSU_{2n}$ & $\lieSp_n$ & $2\leq n$ &$A_{n-1}$& $(n-1)(2n+1)$ & & $4$ & \\
%\hline
\multirow{2}{*}{$\mathsf{A\, III}$}
& $\lieSU_{p+q}$ & $\mathfrak{s}(\lieU_p\oplus\lieU_q)$ & $2\leq p<q$ &$BC_{p}$&
$2pq$ & $2(q-p)$& $2$ & $1$\\
& $\lieSU_{2p}$ & $\mathfrak{s}(\lieU_p\oplus\lieU_p)$ & $2\leq p$ &$C_{p}$&
$2p^2$ & $2$& $1$ & \\
%\hline
\multirow{2}{*}{$\mathsf{BD}$}
 & $\lieO_{2n+1}\times\lieO_{2n+1}$ & $\Delta\lieO_{2n+1}$ &
& $B_n$ & $n(2n+1)$ & $2$ & $2$ & \\
& $\lieO_{2n}\times \lieO_{2n}$ & $\Delta\lieO_{2n}$ & $3\leq n$
& $D_n$ & $n(2n-1)$ & & $2$ & \\
%\hline
\multirow{2}{*}{$\mathsf{BD\, I}$}
 & $\lieO_{p+q}$ & $\lieO_p\oplus \lieO_q$ & $2\leq p<q$
& $B_p$ & $pq$ & $q-p$ & $1$ & \\
& $\lieO_{2p}$ & $\lieO_p\oplus \lieO_p$ & $3\leq p$
& $D_p$ & $p^2$ & & $1$ & \\
%\hline
$\mathsf{C}$& $\lieSp_{n}\times\lieSp_n$ & $\Delta\lieSp_n$ &
& $C_n$ & $n(2n+1)$ &$2$ & $2$ & \\
%\hline
$\mathsf{C\, I}$& $\lieSp_{n}$ & $\lieU_n$ & $2\leq n$
& $C_n$ & $n(n+1)$ &$1$ & $1$ & \\
%\hline
\multirow{2}{*}{$\mathsf{C\, II}$}& $\lieSp_{p+q}$ & $\lieSp_p\oplus\lieSp_q$ & $2\leq p<q$
& $BC_p$ & $4pq$ &$4(q-p)$ & $4$ &$3$ \\
& $\lieSp_{2p}$ & $\lieSp_p\oplus\lieSp_p$ & $2\leq p$
& $C_p$ & $4p^2$ &$4$ & $3$ & \\
%\hline
\multirow{2}{*}{$\mathsf{D\, III}$}& $\lieO_{4n}$ & $\lieU_{2n}$ & $2\leq n$
& $C_{n}$ & $2n(2n-1)$ &$4$ & $1$ & \\
 & $\lieO_{2(2n+1)}$ & $\lieU_{2n+1}$ & $2\leq n$
& $BC_{n}$ & $2n(2n+1)$ &$4$ & $4$ & $1$\\
\end{tabular}
}
\end{table}

%%%%%%%%%%%%%%%%%%%%%%%%%%%%%%%%%%%%%%%%%%%%%%%%%%%%%%%%%%%% Tables with Cartan Symbol
%%%%%%%%%%%%%%%%%%%%%%%%%%%%%%%%%%%%%%%%%%%%%%%%%%%%%%%%%%%

\begin{table}[h]\centering
\caption{$d_P, k_P$ and $C_P$ for classical irreducible higher rank symmetric spaces of compact type}
\label{TAB: type I short}
\Rotatebox{90}{
\begin{tabular}{cccccccc}
%\toprule
CS&$\lieG$ & $\lieK$ &Cond.\
& $d_P$& $k_P$ & $C_P$&$\sharp(Q)$\\
\midrule
$\mathsf{A}$ & $\lieSU_n\times \lieSU_n$ & $\Delta\lieSU_n$ & $3\leq n$ &
$2(n-1)$ &$(n-1)^2$&$n-5$& $2n-2\cdot \Codim(Q)$\\
%\hline
$\mathsf{A\, I}$ & $\lieSU_n$ & $\lieO_n$ & $2\leq n$ &
$n-1$ &$\frac{n}{2}(n-1)$&$\frac{n-9}{2}$& $n+1-2\cdot \Codim(Q)$\\
%\hline
$\mathsf{A\, II}$ & $\lieSU_{2n}$ & $\lieSp_n$ & $2\leq n$ &
$4(n-1)$ &$(2n-3)(n-1)$&$2n-6$& $4n-2-2\cdot \Codim(Q)$\\
%\hline
\multirow{4}{*}{$\mathsf{A\, III}$} & $\lieSU_{4}$ & $\mathfrak{s}(\lieU_2\oplus\lieU_2)$ & &
$4$ &$4$&-2& $6-2\cdot \Codim(Q)$\\
& $\lieSU_{5}$ & $\mathfrak{s}(\lieU_2\oplus\lieU_3)$ & &
$7$ &$5$&-1/2& $9-2\cdot \Codim(Q)$\\
& $\lieSU_{2+q}$ & $\mathfrak{s}(\lieU_2\oplus\lieU_q)$ & $4\leq q$ &
$2q+1$ &$2q-1$&$q-\frac{7}{2}$& $2q+3-2\cdot \Codim(Q)$\\
& $\lieSU_{p+q}$ & $\mathfrak{s}(\lieU_p\oplus\lieU_q)$ & $3\leq p\leq q$ &
$2(p+q)-3$ &$2pq-2(p+q)+3$&$p+q-\frac{11}{2}$& $2(p+q)-1-2\cdot \Codim(Q)$\\
%\hline
\multirow{2}{*}{$\mathsf{BD}$} & $\lieO_{2n+1}\times\lieO_{2n+1}$ & $\Delta\lieO_{2n+1}$ &
$2\leq n$
& $4n-2$ &$n(2n-3)+2$&$2n-5$& $4n-2\cdot \Codim(Q)$\\
& $\lieO_{2n}\times\lieO_{2n}$ & $\Delta\lieO_{2n}$ & $4\leq n$
& $4n-4$ &$n(2n-5)+4$&$2n-6$& $4n-2-2\cdot \Codim(Q)$\\
%\hline
\multirow{4}{*}{$\mathsf{BD\, I}$} & $\lieO_{2+q}$ & $\lieO_2\oplus \lieO_q$ & $3\leq q$
& $q$ &$q$&$\frac{q}{2}-4$& $q+2-2\cdot \Codim(Q)$\\
& $\lieO_{6}$ & $\lieO_3\oplus \lieO_3$ &
& $3$ &$6$&-5/2& $5-2\cdot \Codim(Q)$\\
 & $\lieO_{3+q}$ & $\lieO_3\oplus \lieO_q$ & $4\leq q$
& $q+1$ &$2q-1$&$\frac{q-7}{2}$& $q+3-2\cdot \Codim(Q)$\\
& $\lieO_{p+q}$ & $\lieO_p\oplus\lieO_q$ & $4\leq p\leq q$
& $p+q-2$ &$pq-p-q+2$&$\frac{p+q-10}{2}$& $p+q-2\cdot \Codim(Q)$\\
%\hline
$\mathsf{C}$
 & $\lieSp_{n}\times \lieSp_{n}$ & $\Delta\lieSp_{n}$ & $2\leq n$
& $4n-2$ &$n(2n-3)+2$&$2n-5$ & $4n-2\cdot \Codim(Q)$\\
%\hline
$\mathsf{C\, I}$
 & $\lieSp_{n}$ & $\lieU_n$ & $2\leq n$
& $2n-1$ &$n(n-1)+1$&$n-\frac{9}{2}$& $2n+1-2\cdot \Codim(Q)$\\
%\hline
\multirow{3}{*}{$\mathsf{C\, II}$} & $\lieSp_{4}$ & $\lieSp_2\oplus\lieSp_2$ & &
 $10$ &$6$&1& $12-2\cdot \Codim(Q)$\\
 & $\lieSp_{2+q}$ & $\lieSp_2\oplus\lieSp_q$ & $3\leq q$&
 $4q+3$ &$4q-3$&$2q-\frac{5}{2}$& $4q+5-2\cdot \Codim(Q)$\\
 & $\lieSp_{p+q}$ & $\lieSp_p\oplus\lieSp_q$ & $3\leq p\leq q$&
 $4(p+q)-5$ &$4pq-4p-4q+5$&$2(p+q)-\frac{13}{2}$& $4(p+q)-3-2\cdot \Codim(Q)$\\
%\hline
\multirow{5}{*}{$\mathsf{D\, III}$} & $\lieO_{8}$ & $\lieU_{4}$ &
& $6$ &$6$&-1& $8-2\cdot \Codim(Q)$\\
& $\lieO_{10}$ & $\lieU_{5}$ &
& $13$ &$7$&$\frac{5}{2}$& $15-2\cdot \Codim(Q)$\\
& $\lieO_{12}$ & $\lieU_{6}$ &
& $15$ &$15$&$\frac{7}{2}$& $17-2\cdot \Codim(Q)$\\
 & $\lieO_{14}$ & $\lieU_{7}$ &
& $21$ &$21$&$\frac{13}{2}$& $23-2\cdot \Codim(Q)$\\
 & $\lieO_{2n}$ & $\lieU_{n}$ & $8\leq n$
& $4n-7$ &$n(n-5)+7$&$2n-\frac{15}{2}$& $4n-5-2\cdot \Codim(Q)$\\
% \bottomrule
\end{tabular}
}
 \end{table}

\begin{table}[h]\centering
\caption{By courtesy we list the lowest dimension $d_P$ of a non-trivial
 isotropy orbit for irreducible exceptional symmetric spaces(structural information taken from
\cite[Ch.\ X, \S 6]{He})}
\label{TAB: exceptional}
\begin{tabular}{cccccccccc}
%\toprule
CS&$\lieG$ & $\lieK$ &Dynkin &$\dim(P)$ & $m_s$ & $m_l$ & $m_{xl}$ & $d_P$ & $k_P$ \\
\midrule
$\mathsf{E\, 6}$&$\lieE_6\times \lieE_6$& $\Delta\lieE_6$ & $E_6$ & $78$ & & $2$ && $32$&46 \\
%\hdashline
$\mathsf{E\, 7}$ &$\lieE_7\times \lieE_7$&$\Delta\lieE_7$ & $E_7$ & $133$ && $2$&& $54$&79\\
%\hdashline
$\mathsf{E\, 8}$ &$\lieE_8\times \lieE_8$&$\Delta\lieE_8$ & $E_8$ & $248$ && $2$&& $114$ &134\\
%\hdashline
$\mathsf{E\, I}$ & $\lieE_6$ & $\lieSp_4$ & $E_6$ & $42$ & & 1 &&$16$&26 \\
%\hdashline
$\mathsf{E\, II}$ & $\lieE_6$ & $\lieSU_6\oplus\lieSp_1$ & $F_4$ & $40$ &2 & 1 &&$21$& 19\\
%\hdashline
$\mathsf{E\, III}$ & $\lieE_6$ & $\R\oplus\lieO_{10}$ & $BC_2$ & $32$ & 8& 6 &1 &$21$&11\\
%\hdashline
$\mathsf{E\, IV}$ & $\lieE_6$ & $\lieF_4$ & $A_2$ & $26$ & & 8 &&$16$& 10\\
%\hdashline
$\mathsf{E\, V}$ & $\lieE_7$ & $\lieSU_8$ & $E_7$ & $70$ & & 1 & &$27$&43\\
%\hdashline
$\mathsf{E\, VI}$ & $\lieE_7$ & $\lieO_{12}\oplus\lieSp_1$ & $F_4$ & $64$ &4 & 1 &&$33$& 31\\
%\hdashline
$\mathsf{E\, VII}$& $\lieE_7$ & $\R\oplus\lieE_6$ & $C_3$ & $54$ & 8& 1 &&$27$& 27\\
%\hdashline
$\mathsf{E\, VIII}$& $\lieE_8$ & $\lieO_{16}$ & $E_8$ & $128$ & & 1 &&$57$& 71\\
%\hdashline
$\mathsf{E\, IX}$ & $\lieE_8$ & $\lieE_7\oplus\lieSp_1$ & $F_4$ & $112$ &8 & 1 & &$57$&55\\
%\hdashline
$\mathsf{F\, 4}$ &$\lieF_4\times\lieF_4$& $\Delta\lieF_4$ & $F_4$ & $52$ &$2$&$2$ &&$30$&22 \\
%\hdashline
$\mathsf{F\, I}$ & $\lieF_4$ & $\lieSp_3\oplus\lieSp_1$ & $F_4$ & $28$ &1 & 1 & &$15$&13\\
%\hdashline
$\mathsf{F\, II}$ & $\lieF_4$ & $\lieO_{9}$ & $BC_1$ & $16$ &$8$ & &$7$ &$15$&1\\
%\hdashline
$\mathsf{G\, 2}$ &$\lieG_2\times\lieG_2 $&$\Delta\lieG_2$ & $G_2$ & $14$ &$2$ & $2$ &&$10$&4\\
%\hdashline
$\mathsf{G}$ & $\lieG_2$ & $\lieO_4$ & $G_2$ & $8$ &1 & 1 & &$5$&3\\
%\bottomrule
\end{tabular}
\end{table}

%+++++++++++++++++++++++++++++++++++++++++++++++++++++++++++++++++++++++++++++++++++++++++
%. Homotopy Groups
%+++++++++++++++++++++++++++++++++++++++++++++++++++++++++++++++++++++++++++++++++++++++++

\section{Low dimensional homotopy groups}\label{sec05}

This section is devoted to collecting low dimensional topological information of symmetric spaces. The material is basically well-known, and we may collect it from several sources. Nonetheless, either for lack of references or as a service to the reader we also illustrate how most of the information relevant to us can easily be computed by elementary techniques.

 Tables \ref{mastertable_S}, \ref{mastertable_SH}, \ref{mastertable_E}, \ref{mastertable_USH}, and \ref{mastertable_RG} provide us with information on homotopy groups of simply-connected irreducible symmetric spaces of compact type up to at least degree $10$. Note that Table~\ref{mastertable_SH} gives us the stable homotopy groups of such symmetric spaces other than spheres. Since the stable groups are $\mod 8$-periodic, we only provide the information up to degree $9$. To use this table, one should keep in mind that $\pi_{k}$ for $k\equiv 0\mod 8$ is equal to $\pi_{8}$ and $\pi_{k}$ for $k\equiv 1\mod 8$ is equal to $\pi_{9}$, if $k\geq 2$. Moreover, if there is a space whose stable homotopy ends at degree $9$ at most, then we give the information on the homotopy groups of these spaces up to degree $10$ in Table~\ref{mastertable_USH} and Table~\ref{mastertable_RG}. Furthermore, there are some symmetric spaces in low dimensions which are isometric to another one. In this case, we only give the
information of one of the
isometric spaces. If one of the spaces is a real Grassmannian space other than a sphere, then the information can be found in Table~\ref{mastertable_RG}, otherwise in Table~\ref{mastertable_SH}. Such isometric spaces are given in Table~\ref{mastertable_SI}.

The ``f'' in the tables stands for a finite (possibly zero) abelian group. The term ``r1'' stands for ``rank one'', i.e.~for a group of the form $\zz\oplus T$ where $T$ is a finite abelian torsion group. In the same manner ``r$\geq1$'' stands for rank at least one.

The symbol $\In$ actually means that the group is a subgroup, it does not have to be a direct factor, so $\zz_2\In \zz_4$ would also be denoted by ``$\zz_2\In$''. The symbol $\oplus$, however, stands for a direct factor. Further, $G^{m}$ stands for $\underbrace{G\oplus\ldots\oplus G}_{m \, \text{times}}$, where $G$ is an abelian group.

Information about the isometric symmetric spaces in low dimensions in Table~\ref{mastertable_SI} can be found in \cite[pp.~415, 416]{BCO}.
One can find the information about homotopy groups of spheres, i.e. Table~\ref{mastertable_S}, in for example \cite[p. 384]{Hatcher}. We refer to \cite[pp. 251-252]{Concise} and the reference(s) therein for the complex James numbers and the stable homotopy groups of irreducible symmetric spaces of compact type other than spheres given in Table~\ref{mastertable_SH}. For most unstable homotopy groups in Table~\ref{mastertable_USH} we refer to \cite[pp.~254--266]{Concise}. We could not find an exhaustive list of the unstable homotopy groups of Grassmannian manifolds and the symmetric space $\frac{\SU(n)}{\SO(n)}$, $n\leq 7$ in the literature. In what follows, we briefly explain how we compute the homotopy groups in degrees at most $10$.

Basically, and not surprisingly, the main tool that we use to compute the homotopy groups of a symmetric space $G/H$ is the long exact sequence of homotopy groups corresponding to the fibration
$$H\to G\to G/H.$$
For real Grassmannian spaces, we may use the fibration
$$\SO(p)\to \V_{p}(\mathbb{R}^{p+q})\to \Gr_{p}(\mathbb{R}^{p+q}),$$
as well,
where $\V_{p}(\mathbb{R}^{p+q})$ is a real Stiefel manifold. In particular, since $\pi_i(\SO(2))=0$, for $i\geq 2$, we have that
\[
\pi_i(\Gr(2, q))\cong \pi_i(\V(2, q)), \quad \text{for} \quad i\geq 3.
\]
Note that the real Stiefel manifold $ \V_{p}(\mathbb{R}^{p+q})$ is $(q-1)$-connected and, if $p\geq 2$, we have that
\[
\pi_{q}(\V_{p}(\mathbb{R}^{p+q}))=\left \{ \begin{array}{ll}
\zz& \mbox{if $q$ is even}\\
\zz_{2}& \mbox{if $q$ is odd}. \end{array}\right.
\]

For complex Grassmannian spaces we use the fibration
$$\U(p)\to \V_{p}(\mathbb{C}^{p+q})\to \Gr_{p}(\mathbb{C}^{p+q}),$$
where $\V_{p}(\mathbb{C}^{p+q})$ is a complex Stiefel manifold, which is
$2q$-connected.
We refer the reader to \cite{Gilmore, Saito} for some unstable homotopy groups of Stiefel manifolds.

For the symmetric space $\frac{\SU(n)}{\SO(n)}$, $n\leq 7$, we benefit from the following isomorphisms from \cite[Theorem~5.6]{Puttmann}:
$$\pi_{5}\bigg({\frac{\SU(n)}{\SO(n)}}\bigg)\cong\zz\oplus \pi_{4}(\SO(n)),$$
$$\pi_{6}\bigg({\frac{\SU(n)}{\SO(n)}}\bigg)\cong\pi_{5}(\SO(n)).$$

\subsection{Rational homotopy and cohomology of exceptional symmetric spaces}
In this subsection we provide the grounds for the information in Table~\ref{mastertable_E}, which incorporate the references in the literature in some cases and direct computations in other ones.

From \cite[Theorem 7.12, p.~360]{MT91} and mainly from \cite[p.~132]{Mim66} we collect the information on $\F_4$ and $\G_2$.
The homotopy groups of $\E_6$ and $\E_7$ are collected in \cite[p.~4]{Qua11}. The information on $\E_8$ stems from \cite[Theorem V, p.~995]{BS58}. From there we also cite $\pi_{10}(\G_2)\otimes \zz_3=0$.
In \cite[pp.~4-5]{Qua11} the homotopy groups of $\E_6/(\SO(10)\SO(2)$, $\E_6/\F_4$, $\E_7/\SO(2)\E_6$ and the Cayley plane $\F_4/\Spin(9)$ (see also from \cite[p.~132]{Mim66}) are collected.
The information on $\E_6/(\Sp(4)/\pm \id)$ and $\E_7/(\SU(8)/\pm\id)$ stems from \cite[p.~5]{Qua11} enhanced by our rational computations below.

For the types $\mathsf{FI}$, $\mathsf{EII}$, $\mathsf{EVI}$, $\mathsf{EIX}$ a description is given in \cite[Proposition 2.7, p.~227]{Bur92}.
We draw, however, on computations using the respective fibrations and their long homotopy sequences. This builds upon the homotopy groups of the Lie groups and, in particular, on the ones of $\SU(2)\cong \s^3$. We can also provide rational information obtained below. So we do for type $G$, i.e.~the space $\G_2/\SO(4)$.

In \cite[Proposition 2.6, p.~226]{Bur92} it is shown that $\pi_{i+1}\mathsf{EVIII}\cong \pi_i\Gr_4(\rr^{12})$ for $1\leq i\leq 5$.

\bigskip

After a review of the references, let us now to proceed by some computations. We denote by $i\co H\to G$ the inclusion of Lie groups. Recall that for simple Lie groups this inclusion induces an isomorphism on third rational cohomology.

We now shall provide information on the rational cohomology type and the rational homotopy groups at the same time, as the arguments are interdependent.

As for rational invariants we remark that by a classical theorem by Cartan compact symmetric spaces are formal (actually even geometrically formal, since the product of harmonic forms is harmonic), whence their rational homotopy type is determined by their rational cohomology algebra. In particular, the rational homotopy groups are entirely encoded in rational cohomology.

First we describe the rational homotopy type of the compact Lie groups. From Hopf's Theorem we recall that a connected compact Lie group has the rational homotopy type of a finite product of odd-dimensional spheres $G\simeq_\qq \prod_j \s^{i_j}$. In Table \ref{Ftable02} we cite the dimensions $i_j$ of the respective spheres from \cite[p.~956]{Jam95}.
\begin{table}[h]
\centering \caption{Rational types of Lie groups}
\label{Ftable02}
\begin{tabular}{l@{\hspace{3mm}} | l@{\hspace{3mm}} l }
Dynkin type & degrees of rational homotopy groups \\
\hline $A_n$ & $3,5,\dots ,2n+1$\\
$B_n$ & $3,7, \dots, 4n-1$
\\
$C_n$ & $ 3,7, \dots, 4n-1$\\
 $D_n$ & $3,7, \dots, 4n-5, 2n-1$\\
$E_6$ & $ 3,9,11,15,17,23$\\
$E_7$ & $ 3,11,15,19,23,27,35$\\
$E_8$ & $ 3,15,23, 27, 35, 39, 47, 59$\\
$F_4$ & $ 3,11,15,23$\\
$G_2$ & $ 3,11$\\
\end{tabular}
\end{table}

The construction of a Sullivan model of a homogeneous spaces is elaborately depicted both in the standard works \cite{FHT01} and \cite{FOT08}.
\subsubsection{Type $\mathsf{EI}$}
This space is $\E_6/(\Sp(4)/\pm \id)$. We form its rational model $(H^*(\B \Sp(4))\otimes H^*(\E_6),\dif)$. By degree, it follows that $\pi_{10}(\B i)\otimes \qq=\pi_{18}(\B i)\otimes \qq= \pi_{24}(\B i)\otimes \qq=0$. Hence, as cohomology is finite dimensional, we derive that $\pi_4(\B i)\otimes \qq\neq 0$, $\pi_{12}(\B i)\otimes \qq\neq 0$, $\pi_{16}(\B i)\otimes \qq\neq 0$. It follows that
\begin{align*}
\E_6/(\Sp(4)/\pm \id)\simeq_\qq \operatorname{CaP}^2 \times \s^9 \times \s^{17}
\end{align*}
The Cayley plane has rational cohomology $\qq[x_8]/x_8^3$ generated in degree $8$.

Hence the space has one-dimensional rational homotopy groups in degrees
\begin{align*}
8,9,17,23
\end{align*}

\subsubsection{Type $\mathsf{EII}$}
This space is $\E_6/(\SU(6)\SU(2))$. We form its rational model $(H^*(\B (\SU(6)\SU(2)))\otimes H^*(\E_6),\dif)$. By degree, it follows that $\pi_{10}(\B i)\otimes \qq=0$. We derive that
\begin{align*}
H^{\leq 11}(\E_6/(\SU(6)\SU(2));\qq)=(\qq[x_4,x_6,x_8])^{\leq 11}
\end{align*}
where subscript denotes degree.

The space has one-dimensional rational homotopy groups in degrees
\begin{align*}
4,6,8,15,17,23
\end{align*}
plus potentially additional one-dimensional rational homotopy in degrees $9$ and $10$ respectively $11$ and $12$.

\subsubsection{Type $\mathsf{EIII}$}
This space is $\E_6/(\SO(10)\SO(2))$. We form its rational model $(H^*(\B (\SO(10)\SO(2)))\otimes H^*(\E_6),\dif)$. By degree, it follows that $\pi_{10}(\B i)\otimes \qq=0$. As cohomology is finite dimensional, we derive that $\pi_{16}(\B i)\otimes \qq\neq 0$ and hence $\pi_{12}(\B i)\otimes \qq\neq 0$. Therefore,
\begin{align*}
H^{\leq 17}(\E_6/(\SO(10)\SO(2));\qq)=(\qq[x_1,x_2])^{\leq 17}
\end{align*}
with $\deg x_1=2$, $\deg x_2=8$.

The space has one-dimensional rational homotopy groups in degrees
\begin{align*}
2,8,17,23
\end{align*}

\subsubsection{Type $\mathsf{EIV}$}
This space is $\E_6/\F_4$. We form its rational model \linebreak[4]$(H^*(\B \F_4)\otimes H^*(\E_6),\dif)$. By degree, it follows that $\pi_{10}(\B i)\otimes \qq=\pi_{18}(\B i)\otimes \qq=0$. As cohomology is finite dimensional, we derive that $\pi_{23}(\B i)\otimes \qq\neq 0$ and $\pi_{16}(\B i)\otimes \qq\neq 0$ and hence $\pi_{12}(\B i)\otimes \qq\neq 0$. Consequently,
\begin{align*}
\E_6/\F_4 \simeq_\qq \s^9\times \s^{17}
\end{align*}

The space has one-dimensional rational homotopy groups in degrees
\begin{align*}
9,17
\end{align*}

\subsubsection{Type $\mathsf{EV}$}
This space is $\E_7/(\SU(8)/(\pm \id))$. We form its rational model $(H^*(\B (\SU(8)/(\pm \id)))\otimes H^*(\E_7),\dif)$. We derive that
\begin{align*}
H^{\leq 11}(\E_7/(\SU(8)/(\pm \id));\qq)= (\qq[x_6,x_8,x_{10}])^{\leq 11}
\end{align*}
where subscript denotes degree.

The space has one-dimensional rational homotopy groups in degrees
\begin{align*}
6,8,10,12,14,19,23,27,35
\end{align*}
plus potentially additional one-dimensional rational homotopy in degrees $11$ and $12$ respectively $15$ and $16$. Moreover, it has second homotopy group equal to $\zz_2$ as can easily be derived from the long exact sequence in homotopy.

\subsubsection{Type $\mathsf{EVI}$}
This space is $\E_7/(\SO(12)\SU(2))$. We form its rational model $(H^*(\B(\SO(12)\SU(2)))\otimes H^*(\E_7),\dif)$. As cohomology is finite dimensional, we derive that $\pi_{16}(\B i)\otimes \qq\neq 0$, $\pi_{20}(\B i)\otimes \qq\neq 0$. Consequently, also $\pi_{12}(\B i)\otimes \qq\neq 0$. We deduce that
\begin{align*}
H^{\leq 23}(\E_7/(\SO(12)\SU(2));\qq)=(\qq[x_4,x_8,x_{12}])^{\leq 23}
\end{align*}
where subscript again denotes degree.

The space has one-dimensional rational homotopy groups in degrees
\begin{align*}
4,8,12,23,27,35
\end{align*}
plus potentially additional one-dimensional rational homotopy in degrees $11$ and $12$.

\subsubsection{Type $\mathsf{EVII}$}
This space is $\E_7/(\E_6\SO(2))$. We form its rational model $(H^*(\B (\E_6\SO(2)))\otimes H^*(\E_7),\dif)$. Since cohomology is finite dimensional, we derive that $\pi_{16}(\B i)\otimes \qq\neq 0$, $\pi_{24}(\B i)\otimes \qq\neq 0$. It follows that also $\pi_{12}(\B i)\otimes \qq\neq 0$

By degree, it follows that $\pi_{10}(\B i)\otimes \qq=\pi_{18}(\B i)\otimes \qq=0$. As cohomology is finite dimensional, we derive that $\pi_{16}(\B i)\otimes \qq\neq 0$, $\pi_{20}(\B i)\otimes \qq\neq 0$. Consequently, also $\pi_{12}(\B i)\otimes \qq\neq 0$. We deduce that
\begin{align*}
H^{\leq 19}(\E_7/(\E_6\SO(2));\qq)=(\qq[x_2,x_{10},x_{18}])^{\leq 19}
\end{align*}
where subscript again denotes degree.

The space has one-dimensional rational homotopy groups in degrees
\begin{align*}
2,10,18,19,27,35
\end{align*}

\subsubsection{Type $\mathsf{EVIII}$}
This space is $\E_8/\Spin(16)$. We form its rational model $(H^*(\B \Spin(16))\otimes H^*( \E_8),\dif)$. Since cohomology is finite dimensional, we derive that $\pi_{16}(\B i)\otimes \qq\neq 0$.

We deduce that
\begin{align*}
H^{\leq 23}(\E_8/\Spin(16));\qq)=(\qq[x_8,x_{12},x_{16}, x_{20}])^{\leq 23}
\end{align*}
where subscript again denotes degree.

The space has one-dimensional rational homotopy groups in degrees
\begin{align*}
8,12,16,20,35,39,47,59
\end{align*}
plus potentially additional one-dimensional rational homotopy in degrees $23$ and $24$.

\subsubsection{Type $\mathsf{EIX}$}
This space is $\E_8/(\E_7\SU(2))$. We form its rational model $(H^*(\B (\E_7\SU(2)))\otimes H^*(\E_8),\dif)$. Since cohomology is finite dimensional, we derive that $\pi_{28}(\B i)\otimes \qq\neq 0$, $\pi_{36}(\B i)\otimes \qq\neq 0$.

We deduce that
\begin{align*}
H^{\leq 15}(\E_8/(\E_7\SU(2));\qq)=(\qq[x_4,x_{12}])^{\leq 15}
\end{align*}
where subscript again denotes degree.

The space has one-dimensional rational homotopy groups in degrees
\begin{align*}
4,12,20,39,47,59
\end{align*}
plus potentially additional one-dimensional rational homotopy in degrees $15$ and $16$ respectively $23$ and $24$.

\subsubsection{Type $\mathsf{FI}$}
This space is $\F_4/(\Sp(3)\SU(2))$. We form its rational model $(H^*(\B (\Sp(3)\SU(2)))\otimes H^*(\F_4),\dif)$. Since cohomology is finite dimensional, we derive that $\pi_{28}(\B i)\otimes \qq\neq 0$, $\pi_{36}(\B i)\otimes \qq\neq 0$.

We deduce that
\begin{align*}
H^{\leq 11}(\F_4/(\Sp(3)\SU(2));\qq)=(\qq[x_4,x_{8}])^{\leq 11}
\end{align*}
where subscript again denotes degree.

The space has one-dimensional rational homotopy groups in degrees
\begin{align*}
4,8,15, 23
\end{align*}
plus potentially additional one-dimensional rational homotopy in degrees $11$ and $12$.

\subsubsection{Type $\mathsf{FII}$}
This space is the Cayley plane $\F_4/\Spin(9)$. We form its rational model $(H^*(\B \Spin(9))\otimes H^*(\F_4),\dif)$. Since cohomology is finite dimensional, we derive that $\pi_{12}(\B i)\otimes \qq\neq 0$, $\pi_{16}(\B i)\otimes \qq\neq 0$.

We deduce that
\begin{align*}
H^{*}(\F_4/\Spin(9);\qq)=H^*(\operatorname{CaP}^2;\qq)=\qq[x_8]/x_8^3
\end{align*}
where subscript again denotes degree.

The space has one-dimensional rational homotopy groups in degrees
\begin{align*}
8,23
\end{align*}
plus potentially additional one-dimensional rational homotopy in degrees $11$ and $12$.

\subsubsection{Type $\mathsf{G}$}
This Wolf space is $\G_2/\SO(4)$, obviously rationally an $\hh\pp^2$ with one-dimensional rational homotopy groups in degrees $3$ and $11$.
From the theory of quaternion K\"ahler manifolds we also know that $\pi_2(\G_2/\SO(4))$ is finite with non-trivial two-torsion.

%%%%%%%%%%%%%%%%%%%%%%%%%%%%%%%%%%%%%%%%%%%%%%%%%%%%%%%%%%%% TABLES
%%%%%%%%%%%%%%%%%%%%%%%%%%%%%%%%%%%%%%%%%%%%%%%%%%%%%%%%%%%

{ \centering % Center table
\begin{table}[p]%[htbp]
\caption{\label{mastertable_SI} Special Isomorphisms.}
\begin{center}
\small{
\begin{tabular}{lclclclcl}%\toprule
%\multirow{2}{*}{} &
% \multicolumn{9}{c}{$k$} % \vspace{1mm}
% \\ \cline{2-10}

%\midrule
 $\frac{\SU(2)}{\SO(2)}$& $\cong$& {$\frac{\SO(4)}{\U(2)}$}&$\cong$& $\frac{\Sp(1)}{\U(1)}$& $\cong$&$ \mathbb{CP}^{1} $& $\cong$&{$\s^{2}$}\\
 [.1cm] \mytableextraspace
$\SU(2)$&$\cong$&$\Spin(3)$&$\cong$&$ \Sp(1)$&$\cong$&$\s^{3}$\\
[.1cm] \mytableextraspace
$\Spin(4)$&$\cong$&$\SU(2)\times \SU(2)$\\
[.1cm] \mytableextraspace
$\frac{\Sp(2)}{\Sp(1)\Sp(1)}$&$\cong$&$ \s^{4}$\\
[.1cm] \mytableextraspace
$\frac{\SU(4)}{\Sp(2)}$&$\cong$& {$ \s^{5}$}\\
[.1cm] \mytableextraspace
$\Sp(2)$ &$\cong$& {$\Spin(5)$} \\
[.1cm] \mytableextraspace
$\SU(4)$&$\cong$&$\Spin(6)$&&&&\\
[.1cm] \mytableextraspace
$\frac{\SU(4)}{\Spe(\U(2)\U(2))}$& $\cong$&{$ \frac{\SO(6)}{\SO(2)\SO(4)}$}\\
[.1cm] \mytableextraspace
{$ \frac{\SO(4)}{\SO(2)\SO(2)}$}&$\cong$&$\s^{2}\times \s^{2}$\\
[.1cm] \mytableextraspace
{$ \frac{\SO(5)}{\SO(2)\SO(3)}$}&$\cong$&$\frac{\Sp(2)}{\U(2)}$\\
[.1cm] \mytableextraspace
{$ \frac{\SO(8)}{\SO(2)\SO(6)}$}&$\cong$&$\frac{\SO(8)}{\U(4)}$\\
[.1cm] \mytableextraspace
{$ \frac{\SO(6)}{\SO(3)\SO(3)}$}&$\cong$&$\frac{\SU(4)}{\SO(4)}$\\
[.1cm] \mytableextraspace
$\frac{\SO(6)}{\U(3)}$&$\cong$& {$\mathbb{CP}^{3}$}\\
\end{tabular}
}
\end{center}

\end{table}
}
  \clearpage% Flush page

%%%%%%%%%%%%%%%%%%%%%%%%%%%%%%%%%%%%%%%%%%%%%%%%%%%%%%%%%%%

%\afterpage{%
  % \thispagestyle{empty}% empty page style (?)
  %\begin{landscape}
{ \centering % Center table

\begin{table}[p]%[htbp]
\caption{\label{mastertable_S} Homotopy groups of spheres.}

\Rotatebox{90}{
\scriptsize{
\begin{tabular}{p{28mm}ccccccccccccc}%\toprule
%\multirow{2}{*}{} &
% \multicolumn{9}{c}{$k$} % \vspace{1mm}
% \\ \cline{2-10}
  & $\s^{1}$ & $\s^{2}$& $\s^{3}$ & $\s^{4}$ & $\s^{5}$ &$\s^{6}$ &$\s^{7}$ &$\s^{8}$ &$\s^{9}$&$\s^{10}$ & $\s^{11}$& $\s^{12}$ & $\s^{\geq 13}$
 \\
\midrule
$\pi_{<n}(\s^{n})$ & { $0$}& { $0$ }& { $0$ }& { $0$} &{ $0$ }& { $0$ }& { $0$}& { $0$}& { $0$ }&$0$&$0$&$0$&$0$\\
[.1cm] \mytableextraspace
$\pi_{0+n}(\s^{n})$ & $\zz$&$\zz$ & $\zz$ & $\zz$ &$\zz$ & $\zz$ &$\zz$&$\zz$& $\zz$&$\zz$&$\zz$&$\zz$&$\zz$\\
[.1cm] \mytableextraspace
$\pi_{1+n}(\s^{n})$ & $0$&$\zz$ & $\zz_{2}$ & $\zz_{2}$ &$\zz_{2}$ & $\zz_{2}$ &$\zz_{2}$&$\zz_{2}$& $\zz_{2}$&$\zz_{2}$&$\zz_{2}$&$\zz_{2}$&$\zz_{2}$\\
[.1cm] \mytableextraspace
 $\pi_{2+n}(\s^{n})$ & $0$&$\zz_{2}$ & $\zz_{2}$ & $\zz_{2}$ &$\zz_{2}$ & $\zz_{2}$ &$\zz_{2}$&$\zz_{2}$& $\zz_{2}$&$\zz_{2}$&$\zz_{2}$&$\zz_{2}$&$\zz_{2}$\\
[.1cm] \mytableextraspace
$\pi_{3+n}(\s^{n})$ & $0$&$\zz_{2}$ & $\zz_{12}$ & $\zz\oplus\zz_{12}$ &$\zz_{24}$ & $\zz_{24}$ &$\zz_{24}$&$\zz_{24}$& $\zz_{24}$&$\zz_{24}$&$\zz_{24}$&$\zz_{24}$&$\zz_{24}$\\
[.1cm] \mytableextraspace
$\pi_{4+n}(\s^{n})$ & $0$&$\zz_{12}$ & $\zz_{2}$ & $\zz_{2}^{2}$ &$\zz_{2}$ & $0$ &$0$&$0$& $0$&$0$&$0$&$0$&$0$\\
[.1cm] \mytableextraspace
$\pi_{5+n}(\s^{n})$ & $0$&$\zz_{2}$ & $\zz_{2}$ & $\zz_{2}^{2}$ &$\zz_{2}$ & $\zz$ &$0$&$0$& $0$&$0$&$0$&$0$&$0$\\
[.1cm] \mytableextraspace
$\pi_{6+n}(\s^{n})$ & $0$&$\zz_{2}$ & $\zz_{3}$ & $\zz_{24}\oplus\zz_{3}$ &$\zz_{2}$ & $\zz_{2}$ &$\zz_{2}$&$\zz_{2}$& $\zz_{2}$&$\zz_{2}$&$\zz_{2}$&$\zz_{2}$&$\zz_{2}$\\
[.1cm] \mytableextraspace
$\pi_{7+n}(\s^{n})$ & $0$&$\zz_{3}$ & $\zz_{15}$ & $\zz_{15}$ &$\zz_{30}$ & $\zz_{60}$ &$\zz_{120}$&$\zz\oplus\zz_{120}$& $\zz_{240}$&$\zz_{240}$&$\zz_{240}$&$\zz_{240}$&$\zz_{240}$\\
[.1cm] \mytableextraspace
$\pi_{8+n}(\s^{n})$ & $0$&$\zz_{15}$ & $\zz_{2}$ & $\zz_{2}$ &$\zz_{2}$ & $\zz_{24}\oplus\zz_{2}$ &$\zz_{2}^{3}$&$\zz_{2}^{4}$& $\zz_{2}^{3}$&$\zz_{2}^{2}$&$\zz_{2}^{2}$&$\zz_{2}^{2}$&$\zz_{2}^{2}$\\
[.1cm] \mytableextraspace
$\pi_{9+n}(\s^{n})$ & $0$&$\zz_{2}$ & $\zz_{2}^{2}$ & $\zz_{2}^{3}$ &$\zz_{2}^{3}$ & $\zz_{2}^{3}$ &$\zz_{2}^{4}$&$\zz_{2}^{5}$& $\zz_{2}^{4}$&$\zz\oplus \zz_{2}^{3}$&$\zz_{2}^{3}$&$\zz_{2}^{3}$&$\zz_{2}^{3}$\\
[.1cm] \mytableextraspace
$\pi_{10+n}(\s^{n})$ & $0$&$\zz_{2}^{2}$ & $\zz_{12}\oplus\zz_{2}$ & $\zz_{120}\oplus\zz_{12}\oplus\zz_{2}$ &$\zz_{72}\oplus\zz_{2}$ & $\zz_{72}\oplus\zz_{2}$ &$\zz_{24}\oplus\zz_{2}$&$\zz_{24}^{2}\oplus\zz_{2}$& $\zz_{24}\oplus\zz_{2}$&$\zz_{12}\oplus\zz_{2}$&$\zz_{6}\oplus\zz_{2}$&$\zz_{6}$&$\zz_{6}$\\

%\bottomrule
\end{tabular}
}
}

\end{table}
}
  \clearpage% Flush page
%}

%%%%%%%%%%%%%%%%%%%%%%%%%%%%%%%%%%%%%%%%%%%%%%%%%%%%%%%%%%%

%\afterpage{%
  % \thispagestyle{empty}% empty page style (?)
  %\begin{landscape}
{ \centering % Center table
\begin{table}[p]%[htbp]
\caption{\label{mastertable_SH} Stable homotopy groups of simply-connected classical symmetric spaces of compact type ($p\leq q$).}
\begin{center}
\small{
\begin{tabular}{p{28mm}ccccccccc}%\toprule
\multirow{2}{*}{} &
  \multicolumn{9}{c}{$k$} % \vspace{1mm}
  \\ \cline{2-10}
  & $1$ & $2$& $3$ & $4$ & $5$ &$6$ &$7$ &$8$ &$9$%&$10$
 \\
\midrule
$\pi_{k}(\SU(n))$ & { $0$}& { $0$ }& { $\zz$ }& { $0$} &{ $\zz$ }& { $0$ }& { $\zz$}& { $0$}& { $\zz$ }\\
{$k\leq 2n-1$}&&&&&&&&&\\
[.1cm] \mytableextraspace
$\pi_{k}(\Spin(n))$ & $0$&$0$ & $\zz$ & $0$ &$0$ & $0$ &$\zz$&$\zz_{2}$& $\zz_{2}$\\
{$k\leq n-2$}&&&&&&&&&\\
[.1cm] \mytableextraspace
$\pi_{k}(\Sp(n))$ & $0$&$0$ & $\zz$ & $\zz_{2}$ &$\zz_{2}$ & $0$ &$\zz$&$0$& $0$ \\
 {$k\leq 4n+1$ }&&&&&&&&&\\
[.1cm] \mytableextraspace
 $\pi_{k}(\frac{\SU(n)}{\SO(n)})$ &$0$&$\zz_{2}$ & $\zz_{2}$ & $0$ &$\zz$ & $0$ &$0$&$0$& $\zz$ \\
 $k\leq n-1$ &&&&&&&&&\\
[.1cm] \mytableextraspace
$\pi_{k}(\frac{\SU(2n)}{\Sp(n)})$ & $0$&$0$ & $0$ & $0$ &$\zz$ & $\zz_{2}$ &$\zz_{2}$&$0$& $\zz$ \\
$k\leq 4n-1$ &&&&&&&&&\\
[.1cm] \mytableextraspace
$\pi_{k}(\frac{\SU(p+q)}{\Spe(\U(p)\U(q))})$ & $0$&$\zz$ & $0$ & $\zz$ &$0$ & $\zz$ &$0$&$\zz$& $0$ \\
$p\geq 2, q\geq 5$,\\$k< 2q+1$ &&&&&&&&\\
[.1cm] \mytableextraspace
$\pi_{k}(\frac{\SO(p+q)}{\SO(p)\SO(q)})$ & $0$ & $\zz_2$&$0$ & $\zz$ & $0$ &$0$ &$0$ &$\zz$ &$\zz_2$ \\
$p\geq 11$ $k<q$&&&&&&&&\\
[.1cm] \mytableextraspace
$\pi_{k}(\frac{\SO(2n)}{\U(n)})$ & $0$&$\zz$ & $0$&$0$ & $0$ & $\zz$ &$\zz_{2}$ &$\zz_{2}$ &$0$ \\
$k\leq 2n-2$&&&&&&&&\\
[.1cm] \mytableextraspace
$\pi_{k}(\frac{\Sp(n)}{\U(n)})$ & $0$&$\zz$ & $\zz_{2}$&$\zz_{2}$ & $0$ & $\zz$ &$0$ &$0$ &$0$ \\
$k\leq 2n$&&&&&&&&\\
[.1cm] \mytableextraspace
$\pi_{k}(\frac{\Sp(p+q)}{\Sp(p)\Sp(q)})$ & $0$ & $0$ & $0$ &$\zz$ & $\zz_2$ & $\zz_2$ & $0$ & $\zz$ & $0$ \\
$p\geq 2,$ $k<4q+3$&&&&&&&&%\\
\end{tabular}
}
\end{center}
\end{table}
}
  \clearpage% Flush page
%}

%%%%%%%%%%%%%%%%%%%%%%%%%%%%%%%%%%%%%%%%%%%%%%%%%%%%%%%%%%%

%\input{Exceptional.tex}% \end{landscape}

\begin{table}[p]%[htbp]
\caption{\label{mastertable_E} Homotopy groups of irreducible exceptional symmetric spaces of compact type.}
%\toprule
\Rotatebox{90}{
%\begin{center}
\small{
\begin{tabular}{lccccccccccccccc}
%\toprule
%\multirow{2}{*}{} &
% \multicolumn{15}{c}{$k$} % \vspace{1mm}
% \\ \cline{2-10}
  & $\pi_1$ & $\pi_2$& $\pi_3$ & $\pi_4$ & $\pi_5$ &$\pi_6$ &$\pi_7$ &$\pi_8$ &$\pi_9$&$\pi_{10}$&$\pi_{11}$&$\pi_{12}$&$\pi_{13}$&$\pi_{14}$&$\pi_{15}$
 \\
\midrule
 $\G_2$ & $0$&$0$&$\zz$ &$0$&$0$&$\zz_3$&$0$&$\zz_2$&$\zz_6$ &$0$&$\zz\oplus \zz_2$&$0$&$0$&$\zz_{168}\oplus \zz_2$&$\zz_2$\\
[.1cm] \mytableextraspace
$\F_4$ & $0$&$0$&$\zz$&$0$&$0$&$0$&$0$&$\zz_2$&$\zz_2$ &$0$&$\zz\oplus \zz_2$&$0$&$0$&$\zz_2$&$\zz$ \\
[.1cm] \mytableextraspace
 $\E_6$ &$0$ & $0$ & $\zz$&$0$&$0$&$0$&$0$&$0$&$\zz$ & $0$ & $\zz$ & $\zz_{12}$ & $0$ & $0$ & $\zz$\\
[.1cm] \mytableextraspace
$\E_7$ & $0$ & $0$ & $\zz$&$0$&$0$&$0$&$0$&$0$&$0$ & $0$ & $\zz$ & $\zz_{2}$ & $\zz_2$ & $0$ & $\zz$ \\
[.1cm] \mytableextraspace
$\E_8$ & $0$ & $0$ & $\zz$&$0$&$0$&$0$&$0$&$0$&$0$ & $0$ & $0$ & $0$ & $0$ & $0$ & $\zz$ \\
[.1cm] \mytableextraspace
 $\mathsf{EI}$ & $0$& $\zz_2$&$0$&$0$&$\zz_2$&$\zz_2$&$\zz_2$&$\zz$&$\zz$ &$0$&f&f&f&f&f
\\
[.1cm] \mytableextraspace$\mathsf{EII}$ & $0$&$0$&f&$\zz$&$\zz_2$&$\zz\oplus\zz_2$&$\zz_{12}$&$\zz\oplus \zz_2$& &&&&f&f&r1 \\
[.1cm] \mytableextraspace $\mathsf{EIII}$ &$0$ &$\zz$&$0$&$0$&$0$&$0$&$0$&$\zz$&$\zz_2$&$\zz_2$ &$\zz_{24}$&$0$&$0$&$\zz_2$&$\zz_{120}$
\\
[.1cm] \mytableextraspace $\mathsf{EIV}$ & $0$ & $0$& $0$& $0$& $0$& $0$& $0$& $0$&$\zz$ &$\zz_2$&$\zz_2$&$\zz_{24}$&$0$&$0$&$\zz_2$
\\
[.1cm] \mytableextraspace $\mathsf{EV}$ & $0$& $\zz_2$&$0$&$0$&$0$&$\zz$&$0$&$\zz$&$0$ &$\zz$&&r$\geq 1$&f&r1&\\
\\
[.1cm] \mytableextraspace$\mathsf{EVI}$ & $0$ &$\zz_2$&f&$\zz$&$\zz_2$&$\zz_2$&$\zz_2$&$\zz\oplus \zz_2$&$\zz_2\oplus \zz_2$ &$\zz_6$&&r$\geq 1$&f&f&f
\\
[.1cm] \mytableextraspace $\mathsf{EVII}$ & $0$&$\zz$ &$0$&$0$&$0$&$0$&$0$&$0$&$0$ & $\zz$&$\zz_2$&$\zz_2$&$\zz_{24}$&$0$& $0$
% & $\zz$
\\
[.1cm] \mytableextraspace $\mathsf{EVIII}$ & $0$ &$0$&f&$0$&$0$&$0$&$0$&$\zz$ &$\zz_2$&$\zz_2$&$0$&$\zz$&$0$&$0$&f\\
\\
[.1cm] \mytableextraspace $\mathsf{EIX}$ & $0$&$0$ &f&$\zz$&$\zz_2$&$\zz_2$&$\zz_{12}$&$\zz_2$&$\zz_2$ %&$\zz_3$
\\
[.1cm] \mytableextraspace $\mathsf{FI}$&$0$&$0$ &f&$\zz$&$\zz_2$&$\zz_2$&$\zz_{12}$&$\zz_2$&$\zz_2$ &$\zz_3$&$\zz_{15}$&$\zz\oplus \zz_2$&$\zz_2^3$&$\zz_2^2\oplus \zz_{12}$%&$\zz_3$
\\
[.1cm] \mytableextraspace $\mathsf{FII}$ &$0$&$0$&$0$& $0$&$0$&$0$&$\zz$&$\zz_2$ &$\zz_2$&$\zz_{24}$&$0$&$0$&$\zz_2$& $\zz_{120}$
\\
[.1cm] \mytableextraspace$\mathsf{G}$ & $0$& $\zz_2$&f&$\zz$&$\zz_2^2$&$\zz_2^2\oplus$&$\zz_3\oplus\zz_4^2\oplus$&$\zz_2^2\In$&f &f&r1&f&f&f&f\\

%
%\bottomrule
\end{tabular}
}
}
%\end{center}

\end{table}

  \clearpage% Flush page

%%%%%%%%%%%%%%%%%%%%%%%%%%%%%%%%%%%%%%%%%%%%%%%%%%%%%%%%%%%

 { \centering % Center table
\begin{table}[p]%[htbp]
\begin{center}
\caption{\label{mastertable_USH} Unstable homotopy groups of irreducible classical symmetric spaces of compact type.}
%\Rotatebox{90}{
%\toprule
\scriptsize{
\begin{tabular}{lcccccccccc}
%\toprule
 & $\pi_1$ & $\pi_2$& $\pi_3$ & $\pi_4$ & $\pi_5$ &$\pi_6$ &$\pi_7$ &$\pi_8$ &$\pi_9$ & $\pi_{10}$%&$\pi_{11}$&$\pi_{12}$&$\pi_{13}$&$\pi_{14}$&$\pi_{15}$
 \\
\midrule
$\SU(3)$ & { $0$}& { $0$ }& { $\zz$ }& { $0$} &{ $\zz$ }& { $\zz_{6}$ }& { $0$}& { $\zz_{12}$}& { $\zz_{3}$ }&$\zz_{30}$\\
[.1cm] \mytableextraspace
$\SU(4)$ & $0$ &$0$& $\zz$ & $0$ &$\zz$ &$0$ &$\zz$ &$\zz_{4!}$&$\zz_2$ &$\zz_{120}\oplus \zz_2$\\
[.1cm] \mytableextraspace
$\SU(5)$&{ $0$}& { $0$ }& { $\zz$ }& { $0$} &{ $\zz$ }& { $0$ }& { $\zz$}& { $0$}& { $\zz$ }&$\zz_{120}$\\
[.1cm] \mytableextraspace
$\Spin(3)$& $0$&$0$ & $\zz$ & $\zz_2$ &$\zz_2$ & $\zz_{12}$ &$\zz_2$&$\zz_2$& $\zz_3$ &$Z_{15}$\\
[.1cm] \mytableextraspace
 $\Spin(4)$ & $0$&$0$&$\zz^2$ &$\zz_2^2$ &$\zz_2^2$ & $\zz_{12}^2$ &$\zz_2^2$&$\zz_2^2$& $\zz_3^2$ &$Z_{15}^2$\\
[.1cm] \mytableextraspace
$\Spin(5)$ &$0$ & $0$ & $\zz$ &$\zz_2$ &$\zz_2$ &$0$ &$\zz$ &$0$ & $0$&$\zz_{120}$ \\
[.1cm] \mytableextraspace
 $\Spin(6)$&$0$ &$0$& $\zz$ & $0$ &$\zz$ &$0$ &$\zz$ &$\zz_{4!}$& $\zz_2$ &$\zz_{120}\oplus \zz_2$ \\
 [.1cm] \mytableextraspace
$\Spin(7)$ &$0$ &$0$& $\zz$ & $0$ &$0$ &$0$ &$\zz$ &$\zz_{2}^2$& $\zz_{2}^2$&$\zz_{8}$ \\
[.1cm] \mytableextraspace
 $\Spin(8)$ &$0$ &$0$& $\zz$ & $0$ &$0$ &$0$ &$\zz\oplus\zz$ &$\zz_{2}^3$& $\zz_{2}^3$&$\zz_{8}\oplus\zz_{24}$ \\
 [.1cm] \mytableextraspace
$\Spin(9)$ &$0$ &$0$& $\zz$ & $0$ &$0$ &$0$ &$\zz$ &$\zz_{2}^2$& $\zz_{2}^2$&$\zz_{8}$ \\
[.1cm] \mytableextraspace
$\Spin(10)$&$0$ &$0$& $\zz$ & $0$ &$0$ &$0$ &$\zz$&$\zz_2$&$\zz\oplus\zz_{2}$&$\zz_{4}$\\
[.1cm] \mytableextraspace
$\Spin(11)$&$0$ &$0$& $\zz$ & $0$ &$0$ &$0$ &$\zz$&$\zz_2$&$\zz_{2}$&$\zz_{2}$\\
[.1cm] \mytableextraspace
$\tfrac{\SU(3)}{\SO(3)}$& $0$ & $\zz_2$ & $f$ &$0$ &$\zz\oplus\zz_2$ & $\zz_{2}$ &$f$ &$f$ & $f$ & $f$\\
% Note that $pi_{3} $ and $\pi_{4} $ are both $\zz$ or $pi_{3} =f$ and $\pi_{4} =0$. Then we use Rational Hurewics Theorem and Poincar\'e duality and consult with the cohomology groups. Moreover, $pi_{3} $ is possibly a $2$-torsion group.
 [.1cm] \mytableextraspace
$\tfrac{\SU(4)}{\SO(4)}$& $0$ & $\zz_2$ & $\zz_2$ &$\zz$ &$ \zz\oplus\zz_2^{2} $&$ \zz_2^{2}$ & $r1$& $f$&$f$&$f$ \\
 [.1cm] \mytableextraspace
$\tfrac{\SU(5)}{\SO(5)}$& $0$ & $\zz_2$ & $\zz_2$ & $0$ &$\zz\oplus\zz_{2}$ & $\zz_2$ & $f$& $0$& $\zz$ &$0$\\
 [.1cm] \mytableextraspace
$\tfrac{\SU(6)}{\SO(6)}$& $0$ & $\zz_2$ & $\zz_2$ & $0$ &$\zz$ & $\zz$ &$f$ & $0$&$r1$ &$\zz_{2}$ \\
 [.1cm] \mytableextraspace
$\tfrac{\SU(7)}{\SO(7)}$& $0$ & $\zz_2$ & $\zz_2$ & $0$ &$\zz$ & $0$ & $f$&$0$ & $r1$&$\zz_2^{2}$ \\
 [.1cm] \mytableextraspace
$\tfrac{\SU(8)}{\SO(8)}$ & $0$ & $\zz_2$ & $\zz_2$ & $0$ &$\zz$ & $0$ & $0$&$\zz$ & $\zz\oplus\zz_{2}^{2}$ & $\zz_{8}^{2}\oplus \zz_{2}$ \\
 [.1cm] \mytableextraspace
 $\tfrac{\SU(9)}{\SO(9)}$ & $0$ & $\zz_2$ & $\zz_2$ &$0$ & $\zz$ & $0$ & $0$ & $0$ & $\zz\oplus \zz_{2}$&$\zz\oplus\zz_{2}$\\
 [.1cm] \mytableextraspace
$\tfrac{\SU(10)}{\SO(10)}$ & $0$ & $\zz_2$ & $\zz_2$ &$0$ & $\zz$ & $0$ & $0$ & $0$ & $\zz$&$\zz\oplus\zz_{2}$\\
[.1cm] \mytableextraspace
$\tfrac{\SU(5)}{\Spe(\U(2)\U(3))}$ & $0$ & $\zz$ & $0$ &$\zz$ & $\zz_{2}$ & $\zz_{2}$ & $r1$ & $0$ & $r1$&$f$\\
[.1cm] \mytableextraspace
$\tfrac{\SU(6)}{\Spe(\U(2)\U(4))}$ & $0$ & $\zz$ & $0$ &$\zz$ & $\zz_{2}$ & $\zz_{2}$ & $\zz_{12}$ & $\zz_{2}$ & $r1$&$\zz_{3}\oplus \zz_{2}$\\
[.1cm] \mytableextraspace
$\tfrac{\SU(6)}{\Spe(\U(3)\U(3))}$ & $0$ & $\zz$ & $0$ &$\zz$ & $0$ & $\zz$ & $r1$ & $0$ & $r1$&$\zz_{3}\oplus \zz_{3}$\\
[.1cm] \mytableextraspace
$\tfrac{\SU(7)}{\Spe(\U(3)\U(4))}$ & $0$ & $\zz$ & $0$ &$\zz$ & $0$ & $\zz$ & $\zz_{6}$ & $0$ & $r1$&$\zz_{3}\oplus \zz_{2}$\\
[.1cm] \mytableextraspace
$\tfrac{\SU(8)}{\Spe(\U(4)\U(4))}$ & $0$ & $\zz$ & $0$ &$\zz$ & $0$ & $\zz$ & $0$ & $\zz$ & $r1$&$\zz_{2}\oplus \zz_{2}$\\
[.1cm] \mytableextraspace
$\frac{\SO(10)}{\U(5)}$ & $0$&$\zz$ & $0$&$0$ & $0$ & $\zz$ &$\zz_{2}$ &$\zz_{2}$ &$\zz_{24}$&$0$ \\
[.1cm] \mytableextraspace
$\frac{\Sp(3)}{\U(3)}$ & $0$&$\zz$ & $\zz_{2}$&$\zz_{2}$ & $0$ & $\zz$ &$\zz$ &$0$ &$\zz_{12}$&$\zz_{3}$ \\
[.1cm] \mytableextraspace
$\frac{\Sp(4)}{\U(4)}$ & $0$&$\zz$ & $\zz_{2}$&$\zz_{2}$ & $0$ & $\zz$ &$0$ &$0$ &$\zz_{24}$&$\zz_{2}$ \\
[.1cm] \mytableextraspace
$\frac{\Sp(1+q)}{\Sp(1)\Sp(q)}$, $q\geq 2$ & $0$ & $0$ & $0$ &$\zz$ & $\zz_2$ & $\zz_2$ & $\zz_{12}$ & $\zz_{2}$ & $\zz_{2}$&$\zz_{3}$ \\
%\bottomrule
\end{tabular}
}
\end{center}
%}
\end{table}
}
 \clearpage

%%%%%%%%%%%%%%%%%%%%%%%%%%%%%%%%%%%%%%%%%%%%%%%%%%%%%%%%%%%
{
 \centering % Center table

\begingroup

%\scriptsize%{
\tiny

\begin{longtable}{lcccccccccc}
\caption{Unstable homotopy groups of Real Grassmannian.}\label{mastertable_RG}\\
%\scriptsize{
%\toprule
 & $\pi_1$ & $\pi_2$& $\pi_3$ & $\pi_4$ & $\pi_5$ &$\pi_6$ &$\pi_7$ &$\pi_8$ &$\pi_9$ & $\pi_{10}$%&$\pi_{11}$&$\pi_{12}$&$\pi_{13}$&$\pi_{14}$&$\pi_{15}$
 \\
\midrule
$\tfrac{\SO(4)}{\SO(2)\times \SO(2)}$& $0$ & $\zz^{2}$ & $\zz^{2}$ & $\zz_{2}^{2}$ & $ \zz_{2}^{2}$& $\zz_{12}^{2}$ &$\zz_{2}^{2}$&$\zz_{2}^{2}$&$\zz_{3}^{2}$&$\zz_{15}^{2}$\\
[.1cm] \mytableextraspace
$\tfrac{\SO(5)}{\SO(2)\times \SO(3)}$ & $0$&$\zz$ & $\zz_{2}$&$\zz_{2}$ & $\zz_{4}$ & $\zz_{2}$ &$\zz\oplus\zz_{2}$ &$\zz_{2}$ &$0$ &$\zz_{3}$\\
[.1cm] \mytableextraspace
$\tfrac{\SO(6)}{\SO(2)\times \SO(4)}$& $0$ & $\zz$ & $0$ & $\zz$ & $\zz\oplus \zz_{2}$& $\zz_2^2$ &$r1$&$f$&$f$&$f$\\
[.1cm] \mytableextraspace
$\tfrac{\SO(7)}{\SO(2)\times \SO(5)}$& $0$ & $\zz$ & $0$ & $0$ &$\zz_{2}$& $\zz_{2}$ &$\zz_4$ &$\zz_2^2$&$f$&$f$ \\
[.1cm] \mytableextraspace
$\tfrac{\SO(8)}{\SO(2)\times \SO(6)}$& $0$&$\zz$ & $0$&$0$ & $0$ & $\zz$ &$\zz\oplus\zz_{2}$ &$\zz_{2}^{2}$ &$\zz_{2}\oplus \zz_{24}$& $\zz_{8}\oplus \zz_{3}$
 \\
 [.1cm] \mytableextraspace
$\tfrac{\SO(9)}{\SO(2)\times \SO(7)}$& $0$ & $\zz$ & $0$ & $0$ &$0$& $0$ & $\zz_2$ &$\zz_{2}$&$\zz_{4}$&$\zz_{2}^{2}$ \\
[.1cm] \mytableextraspace
$\tfrac{\SO(10)}{\SO(2)\times \SO(8)}$& $0$ & $\zz$ & $0$ & $0$ &$0$& $0$ & $0$ &$\zz$& $\zz\oplus \zz_2$& $\zz_2^2 $ \\
[.1cm] \mytableextraspace
$\tfrac{\SO(11)}{\SO(2)\times \SO(9)}$& $0$ & $\zz$ & $0$ & $0$ &$0$& $0$ & $0$ &$0$& $\zz_2$&$\zz_{2}$ \\
[.1cm] \mytableextraspace
$\tfrac{\SO(12)}{\SO(2)\times \SO(10)}$& $0$ & $\zz$ & $0$ & $0$ &$0$& $0$ & $0$ &$0$& $0$& $\zz$ \\
[.1cm] \mytableextraspace
$\tfrac{\SO(2+q)}{\SO(2)\times \SO(q)}$, $q\geq 11 $& $0$ & $\zz$ & $0$ & $0$ &$0$& $0$ & $0$ &$0$& $0$& $0$ \\
[.1cm] \mytableextraspace
$\tfrac{\SO(6)}{\SO(3)\times \SO(3)}$&$0$ & $\zz_2$ & $\zz_2$ &$\zz$ &$ \zz\oplus\zz_2^{2} $&$ \zz_2^{2}$ & $r1$& $f$&$f$&$f$ \\
[.1cm] \mytableextraspace
$\tfrac{\SO(7)}{\SO(3)\times \SO(4)}$& $0$ & $\zz_2$ & $0$ & $\zz^{2}$ &$\zz_{2}^{3}$ &$\zz_2^3$ &$r1$&$f$&$f$&$f$ \\
[.1cm] \mytableextraspace
$\tfrac{\SO(8)}{\SO(3)\times \SO(5)}$& $0$ & $\zz_2$ & $0$ & $\zz$ &$\zz_2^2$&$\zz_2^2$& $r1$&$f$&$f$&$f$ \\
[.1cm] \mytableextraspace
$\tfrac{\SO(9)}{\SO(3)\times \SO(6)}$& $0$ & $\zz_2$ & $0$ & $\zz$ &$\zz_2$&$\zz_2\oplus \zz$ &$f$ &$f$&$f$&$f$ \\
[.1cm] \mytableextraspace
$\tfrac{\SO(10)}{\SO(3)\times \SO(7)}$& $0$ & $\zz_2$ & $0$ & $\zz$ &$\zz_2$&$\zz_2 $ &$f$ &$\zz_{2}$& $r1$&$f$\\
[.1cm] \mytableextraspace
$\tfrac{\SO(11)}{\SO(3)\times \SO(8)}$& $0$ & $\zz_2$ & $0$ & $\zz$ &$\zz_2$&$\zz_2 $ &$\zz_{12}$ &$f$&$f$&$f$ \\
[.1cm] \mytableextraspace
$\tfrac{\SO(12)}{\SO(3)\times \SO(9)}$& $0$ & $\zz_2$ & $0$ & $\zz$ &$\zz_2$&$\zz_2 $ &$\zz_{12}$ &$\zz_2$& $f$ &$f$\\
[.1cm] \mytableextraspace
$\tfrac{\SO(13)}{\SO(3)\times \SO(10)}$& $0$ & $\zz_2$ & $0$ & $\zz$ &$\zz_2$&$\zz_2 $ &$\zz_{12}$ &$\zz_2$& $\zz_3$&$r1$\\
[.1cm] \mytableextraspace
$\tfrac{\SO(3+q)}{\SO(3)\times \SO(q)}$, $q\geq 11 $& $0$ & $\zz_2$ & $0$ & $\zz$ &$\zz_2$&$\zz_2 $ &$\zz_{12}$ &$\zz_2$& $\zz_3$&$\zz_3$\\
[.1cm] \mytableextraspace
$\tfrac{\SO(8)}{\SO(4)\times \SO(4)}$& $0$ & $\zz_2$ & $0$ & $\zz^3$ & $\zz_{2}^{4}$& $\zz_{2}^{4}$&$r\geq 1$ &$f$&$f$&$f$ \\
[.1cm] \mytableextraspace
$\tfrac{\SO(9)}{\SO(4)\times \SO(5)}$& $0$ & $\zz_2$ & $0$ & $\zz^2$ &$\zz_2^3$& $\zz_2^3$&$f$ &$f$& $f$&$f$\\
[.1cm] \mytableextraspace
$\tfrac{\SO(10)}{\SO(4)\times \SO(6)}$& $0$ & $\zz_2$ & $0$ & $\zz^2$ &$\zz_2^2$&$\zz_2^2\oplus \zz$ &$f$ &$f$&$r1$&$f$ \\
[.1cm] \mytableextraspace
$\tfrac{\SO(11)}{\SO(4)\times \SO(7)}$& $0$ & $\zz_2$ & $0$ & $\zz^2$ &$\zz_2^2$&$\zz_{2}^2$ &$f$ &$\zz_2^2$& $\zz_2^3$&$\zz_2^2\oplus\zz_3^2$\\
[.1cm] \mytableextraspace
$\tfrac{\SO(12)}{\SO(4)\times \SO(8)}$& $0$ & $\zz_2$ & $0$ & $\zz^2$ &$\zz_2^2$&$\zz_2^2$&$\zz_{12}^2$ & $f$&$r1$&$f$ \\
[.1cm] \mytableextraspace
$\tfrac{\SO(13)}{\SO(4)\times \SO(9)}$& $0$ & $\zz_2$ & $0$ & $\zz^2$ &$\zz_2^2$&$\zz_2^2$&$\zz_{12}^2$ &$ \zz_2^2$&$f$&$f$ \\
[.1cm] \mytableextraspace
$\tfrac{\SO(14)}{\SO(4)\times \SO(10)}$& $0$ & $\zz_2$ & $0$ & $\zz^2$ &$\zz_2^2$&$\zz_2^2 $ &$\zz_{12}^2$ &$\zz_2^2$& $\zz_3^2$&$r1$\\
[.1cm] \mytableextraspace
$\tfrac{\SO(4+q)}{\SO(4)\times \SO(q)}$, $q\geq 11 $& $0$ & $\zz_2$ & $0$ & $\zz^2$ &$\zz_2^2$&$\zz_2^2 $ &$\zz_{12}^2$ &$\zz_2^2$& $\zz_3^2$&$\zz_3^2$\\
[.1cm] \mytableextraspace
$\tfrac{\SO(10)}{\SO(5)\times \SO(5)}$& $0$ & $\zz_2$ & $0$ & $\zz$ &$\zz_2^2$ &$\zz_2^2$ &$f$ &$r1$&$\zz\oplus\zz_{2}$&$f$\\
[.1cm] \mytableextraspace
$\tfrac{\SO(11)}{\SO(5)\times \SO(6)}$& $0$ & $\zz_2$ & $0$ & $\zz$ &$\zz_2$ & $\zz_2\oplus \zz$ & $f$&$r1$&$\zz_{12}$ &$f$\\
[.1cm] \mytableextraspace
$\tfrac{\SO(12)}{\SO(5)\times \SO(7)}$& $0$ & $\zz_2$ & $0$ & $\zz$ &$\zz_2$ & $\zz_2$ &$f$ &$\zz$&$\zz_2$&$f$ \\
[.1cm] \mytableextraspace
$\tfrac{\SO(13)}{\SO(5)\times \SO(8)}$& $0$ & $\zz_2$ & $0$ & $\zz$ &$\zz_2$ & $\zz_2$ &$0$ &$\zz^2$&$\zz_2^2$&$f$ \\
[.1cm] \mytableextraspace
$\tfrac{\SO(14)}{\SO(5)\times \SO(9)}$& $0$ & $\zz_2$ & $0$ & $\zz$ &$\zz_2$ & $\zz_2$ &$0$ &$\zz$&$\zz_2$&$f$ \\
[.1cm] \mytableextraspace
$\tfrac{\SO(15)}{\SO(5)\times \SO(10)}$& $0$ & $\zz_2$ & $0$ & $\zz$ &$\zz_2$ & $\zz_2$ &$0$ &$\zz$&$0$&$r1$\\
[.1cm] \mytableextraspace
$\tfrac{\SO(5+q)}{\SO(5)\times \SO(q)}$, $q\geq 11 $& $0$ & $\zz_2$ & $0$ & $\zz$ &$\zz_2$ & $\zz_2$ &$0$ &$\zz$&$0$&$0$\\
[.1cm] \mytableextraspace
$\tfrac{\SO(12)}{\SO(6)\times \SO(6)}$& $0$ & $\zz_2$ & $0$ & $\zz$ &$0$ & $\zz^2$ & $f$&$r1$&$f$&$f$ \\
[.1cm] \mytableextraspace
$\tfrac{\SO(13)}{\SO(6)\times \SO(7)}$& $0$ & $\zz_2$ & $0$ & $\zz$ &$0$ & $\zz$ &$f$ &$\zz$&$f$&$f$ \\
[.1cm] \mytableextraspace
$\tfrac{\SO(14)}{\SO(6)\times \SO(8)}$& $0$ & $\zz_2$ & $0$ & $\zz$ &$0$ & $\zz$ & $0$ &$\zz^2$&$f$ &$f$\\
[.1cm] \mytableextraspace
$\tfrac{\SO(15)}{\SO(6)\times \SO(9)}$& $0$ & $\zz_2$ & $0$ & $\zz$ &$0$ & $\zz$ & $0$ &$\zz$&$f$ &$f$\\
[.1cm] \mytableextraspace
$\tfrac{\SO(16)}{\SO(6)\times \SO(10)}$& $0$ & $\zz_2$ & $0$ & $\zz$ &$0$ & $\zz$ & $0$ &$\zz$&$\zz_{24}$&$r1$\\
[.1cm] \mytableextraspace
$\tfrac{\SO(6+q)}{\SO(6)\times \SO(q)}$, $q\geq 11 $& $0$ & $\zz_2$ & $0$ & $\zz$ &$0$ & $\zz$ & $0$ &$\zz$&$\zz_{24}$&$\zz_{2}$\\
[.1cm] \mytableextraspace
$\tfrac{\SO(14)}{\SO(7)\times \SO(7)}$& $0$ & $\zz_2$ & $0$ & $\zz$ &$0$ & $0$ &$f$ &$\zz$&$f$&$f$ \\
[.1cm] \mytableextraspace
$\tfrac{\SO(15)}{\SO(7)\times \SO(8)}$& $0$ & $\zz_2$ & $0$ & $\zz$ &$0$ & $0$ &$0$ &$\zz^2$&$f$&$f$ \\
[.1cm] \mytableextraspace
$\tfrac{\SO(16)}{\SO(7)\times \SO(9)}$& $0$ & $\zz_2$ & $0$ & $\zz$ &$0$ & $0$ &$0$& $\zz$&$f$&$f$ \\
[.1cm] \mytableextraspace
$\tfrac{\SO(17)}{\SO(7)\times \SO(10)}$& $0$ & $\zz_2$ & $0$ & $\zz$ &$0$ & $0$ &$0$& $\zz$& $\zz_2^2$&$r1$\\
[.1cm] \mytableextraspace
$\tfrac{\SO(7+q)}{\SO(7)\times \SO(q)}$, $q\geq 11 $& $0$ & $\zz_2$ & $0$ & $\zz$ &$0$ & $0$ &$0$& $\zz$& $\zz_2^2$&$\zz_2^2$\\
[.1cm] \mytableextraspace
$\tfrac{\SO(16)}{\SO(8)\times \SO(8)}$& $0$ & $\zz_2$ & $0$ & $\zz$ &$0$ & $0$ &$0$ &$\zz^3$&$f$ &$f$\\
[.1cm] \mytableextraspace
$\tfrac{\SO(17)}{\SO(8)\times \SO(9)}$& $0$ & $\zz_2$ & $0$ & $\zz$ &$0$ & $0$ &$0$& $\zz$&$f$&$f$ \\
[.1cm] \mytableextraspace
$\tfrac{\SO(18)}{\SO(8)\times \SO(10)}$& $0$ & $\zz_2$ & $0$ & $\zz$ &$0$ & $0$ &$0$& $\zz$& $\zz_2^3$&$r1$\\
[.1cm] \mytableextraspace
$\tfrac{\SO(8+q)}{\SO(8)\times \SO(q)}$, $q\geq 10 $& $0$ & $\zz_2$ & $0$ & $\zz$ &$0$ & $0$ &$0$& $\zz$& $\zz_2^3$&$\zz_2^3$\\
[.1cm] \mytableextraspace
$\tfrac{\SO(18)}{\SO(9)\times \SO(9)}$& $0$ & $\zz_2$ & $0$ & $\zz$ &$0$ & $0$ &$0$& $\zz$&$f$&$f$ \\
[.1cm] \mytableextraspace
$\tfrac{\SO(19)}{\SO(9)\times \SO(10)}$& $0$ & $\zz_2$ & $0$ & $\zz$ &$0$ & $0$ &$0$& $\zz$& $\zz_2^2$&$r1$\\
[.1cm] \mytableextraspace
$\tfrac{\SO(9+q)}{\SO(9)\times \SO(q)}$, $q\geq 11 $& $0$ & $\zz_2$ & $0$ & $\zz$ &$0$ & $0$ &$0$& $\zz$& $\zz_2^2$&$\zz_2^2$\\
[.1cm] \mytableextraspace
$\tfrac{\SO(20)}{\SO(10)\times \SO(10)}$& $0$ & $\zz_2$&$0$ & $\zz$ & $0$ &$0$ &$0$ &$\zz$ &$\zz_2$ &$\zz^{2}\oplus\zz_{2}$\\
[.1cm] \mytableextraspace
$\tfrac{\SO(10+q)}{\SO(10)\times \SO(q)}$, $q\geq 11$& $0$ & $\zz_2$&$0$ & $\zz$ & $0$ &$0$ &$0$ &$\zz$ &$\zz_2$ &$\zz_{2}$%\\
\end{longtable}
%\end{landscape}
%\end{center}
%}
\endgroup

}
\clearpage

%+++++++++++++++++++++++++++++++++++++++++++++++++++++++++++++++++++++++++++++++++++++++++
% Proof of main Theorems
%+++++++++++++++++++++++++++++++++++++++++++++++++++++++++++++++++++++++++++++++++++++++++

\section{Proof of the main Theorems}\label{S:Proof of main Theorems}
In this section we prove our main theorems. To begin, let $P^{d}$ be an irreducible simply-connected compact classical symmetric space of dimension $d$ and let $k_{P}$ and $ C_{P}$ be as in Table \ref{TAB: type I short}.

First we need the following lemma.

\begin{lemma}\label{L:Focal_Radius}
Let $P$ be an irreducible simply-connected symmetric space of
compact type with $\Ric_{k_{P}}\geq k_{P}$ and let $Q$ be any submanifold of $P$ with $\codim Q\leq C_{P}$. Then $\foc_{Q}\leq \pi/2$.
\end{lemma}

\begin{proof}
A direct computation shows that $\dim Q\geq \dim P-C_{P}\geq k_{P}$. Then by Proposition \ref{P:Focal_Radius}, $\foc_{Q}\leq \pi/2$.
\end{proof}

\begin{proof}[Proof of Theorem~\ref{Main_submanifolds_shape_Operator}]
Let $Q^{l}$ be an $l$-dimensional submanifold of $P^{d}$ satisfying the conditions in the statement of the theorem.
By rescaling the metric with the scalar factor $\lambda=\frac{\delta}{k_{P}}$, we get that $\Ric_{k_{P}}\geq k_{P}$. Let us denote the geometric notions corresponding to the rescaled metric with the superscript $\lambda$. Then $Q$ satisfies the following conditions:
 \begin{itemize}
 \item
$\foc^{\lambda}_{Q}=\sqrt {\lambda}\,\foc_{Q}>\sqrt{\lambda}r\geq 0$. Since $\codim Q\leq C_{P}$, by Lemma~\ref{L:Focal_Radius}, $\foc^{\lambda}_{Q}\leq \pi/2$. Hence $\sqrt{\lambda}r\in[0, \pi/2)$.
\item For every unit normal vector $v$ (with respect to the rescaled metric), and every $k_{P}$-dimensional subspace $W$ of $TQ$, we have that
\begin{align*}
\lvert \trace(S^{\lambda}_{v}\mid_{ W})\rvert &=\frac{1}{\sqrt {\lambda}} \lvert \trace(S_{\sqrt{\lambda}v}\mid_{W})\rvert \\
&\leq \frac{1}{\sqrt {\lambda}}\, {\sqrt {\lambda}} \, k_{P} \cot (\pi/2-\sqrt{\lambda}r)\\
&=k_{P}\cot (\pi/2-\sqrt{\lambda}r).
\end{align*}
 \end{itemize}

Therefore, by Theorem~\ref{T_Connectivity_Principle}, the inclusion map $Q\hto{} P$ is $(2l-d-k_{P}+2)$-connected, or equivalently $\sharp_P(Q)$-connected, where $\sharp_P(Q)$ is the quantity in Table~\ref{TAB: type I short}.

By assumption $\Codim~Q=d-l\leq C_{P}$ which implies that
\begin{equation}\label{EQ:CP}
\sharp_P(Q)=2+d_{P}-2\cdot\Codim(Q)\geq 2+d_{P}-2C_{P}.
\end{equation}
Now our proof has two steps:

\step{1}
 First, for each symmetric space $P$, we apply the corresponding quantities $d_{P}$ and $C_{P}$ from Table~\ref{TAB: type I short} and we conclude that the inclusion map $Q\hto{} P$ is $10$-connected in any case. Therefore, $P$ and $Q$ have isomorphic homotopy groups up to degree at least $9$.
This, in particular, implies that $Q$ is simply-connected and, since, by assumption, $Q$ is isomorphic to a symmetric space of compact type, we deduce that $Q$ is isomorphic to a Riemannian product

\step{2} In the second step, first we analyze the homotopy groups of simply-connected irreducible symmetric spaces of compact type from Tables~\ref{mastertable_S}, \ref{mastertable_SH}, \ref{mastertable_E}, \ref{mastertable_USH}, and \ref{mastertable_RG}. These tables reveal the following information:
\begin{itemize}
\item[(i)] Among simply-connected irreducible symmetric spaces of compact type, only spheres of dimensions at least $10$ have trivial homotopy groups up to degree $9$. This shows that
only
 products with spheres up to dimensions at least $10$ do not change the first $9$ homotopy groups of the symmetric spaces in question.
\item[(ii)] $\pi_{j}(\tfrac{\SO(2+q)}{\SO(2)\SO(q)})=\pi_{j}(\mathbb{CP}^{n})$, for $0\leq j\leq 9$, $q\geq 10$, and $n\geq 5$. This shows that one cannot discern the real Grassmannian $\tfrac{\SO(2+q)}{\SO(2)\SO(q)}$, $q\geq 10$, and complex projective spaces $\mathbb{CP}^{n}$, $n\geq 5$, by their first $9$ homotopy groups.
\item[(iii)] For other irreducible symmetric spaces, if they have the same homotopy groups up to degree $9$, then they have the same Cartan type. In particular, if a reducible symmetric space $X_{1}\times\ldots \times X_{r}$ has the same homotopy groups up to degree $9$ as an irreducible symmetric space $P$, then there is exactly one factor, say, without restriction, $X_{1}$, which has the same Cartan type as $P$. The other factors $X_{\geq 2}$ then necessarily have trivial homotopy groups up to degree $9$, and, consequently, are isomorphic to spheres by Part (i).
\end{itemize}

From this information, one can now prove Parts 1, 2, and 3 of the theorem. For Part 4, note that if $Q\approx X$, where $X$ is an irreducible symmetric space of compact type, then since $X$ and $P$ have the same Cartan type, a dimension computation in each case shows that $\codim~Q>C_{P}$. This implies that $X$ must necessarily be reducible.
\end{proof}

The reader unwilling to check the different cases in the proofs by hand may successively formulate these comparisons of homotopy groups in the form of systems of linear equations. Indeed, successively picking different coefficients, starting with rational ones, then using $\zz_2$-coefficients, $\zz_3$-coefficients, etc., one uses the additivity of homotopy groups in direct products in order to produce separate systems of linear equations: the variables correspond to the number of Euclidean factors of a fixed type, their coefficient is the rank (over the corresponding ring) of the homotopy group in the considered degree of the respective type, the result equals the rank of the homotopy group in this degree of the ambient space, the number of equations is $9$, one equation for each potentially non-trivial homotopy group. One after the other, these systems can efficiently be solved by computer algebra programmes like \textsf{Mathematica}---as confirmed by the authors.

We remark further that this procedure then as well is astonishingly efficient in order to deal with non-irreducible ambient spaces. That is, several recognition theorems of the form of Corollary \ref{cor01} may easily be extended to and formulated for ambient spaces being direct products to a reasonable extent. We leave a detailed investigation of this fact to the interested reader.

%\begin{rem}
\begin{proof}[Proof of Corollary \ref{cor01}]
\label{R:proof_cor}
One can easily deduce Corollary~\ref{cor01} from Step 2 of the proof of Theorem~\ref{Main_submanifolds_shape_Operator}.
 \end{proof}
\vspace{3mm}
 Now we prove Theorem~\ref{Main_T_g_Range}.

\begin{proof}[Proof of Theorem~\ref{Main_T_g_Range}]
Since $Q$ is a complete totally geodesic embedded submanifold of $P$, it satisfies the hypotheses of Theorem~\ref{Main_submanifolds_shape_Operator}. Therefore, by Part (3) of the theorem, $Q$ has the same Cartan type as $P$, possibly up to products with spheres of dimensions at least $10$. Now we show that $Q$ cannot have a sphere factor, and, as a result, $Q$ has the same Cartan type as $P$.

By contradiction, assume that
\begin{equation}\label{Eq_Decomposition_Q}
Q=Q_{1}\times S^{l_{1}}\times \ldots \times S^{l_{t}},
\end{equation}
 with $t\geq 1$, $l_{i}\geq 10$, $i=1, \ldots, t$, where $Q_{1}$ has the same Cartan type as $P$. Then by \cite[Corollary~6.6]{ChenNagano}, we have that $(x, Q) $ is a polar-meridian pair for $Q$, for $x\in Q$. By Lemma~\ref{L_Inclusion_Pairs}, there exists a polar-meridian pair $(P_+, P_-)$ for $P$ such that $Q$ is a subspace of $P_-$. Now we divide our proof into three cases according to whether $P$ is a complex, quaternionic or real Grassmannian.

First assume that $P$ is a complex Grassmannian manifold $\Gr_{\mathbb{C}}(p, n)$, where $p+q=n$. Then by the classification of polar-meridian pairs for compact irreducible symmetric spaces (e.g.~\cite[pp. 41, 42]{Sumi}), and, since $p<q$, there exist $a, b$, with $a+b=p$, and $0\leq a< p$, such that
\[
Q\subseteq \Gr_{\mathbb{C}}(a, n-2b)\times \Gr_{\mathbb{C}}(b, 2b)=P_-.
\]
  This implies that $$\codim~P_-\leq \codim~Q\leq C_{P}=p+q-11/2.$$
Since $P_{-}$ is also a complete totally geodesic embedded submanifold of $P$, by Theorem~\ref{Main_submanifolds_shape_Operator}, possibly up to a product with spheres, $P_{-}$ has the same Cartan type as $\Gr_{\mathbb{C}}(p, n)$.
This can only occur if $a=0$ and $b=p$, but in this case, $\codim~P_{-}=2pq-2p^{2}$, which is greater than $p+q-11/2$
since $q\geq p+1$. Hence we get a contradiction.

  Now assume that $P$ is a quaternionic Grassmannian manifold $\Gr_{\mathbb{H}}(p, n)$, where $p+q=n$. Similarly, there exist $a, b$, with $a+b=p$, and $0\leq a< p$, such that
  \[
Q\subseteq \Gr_{\mathbb{H}}(a, n-2b)\times \Gr_{\mathbb{H}}(b, 2b)=P_-.
\]
  Again, this implies that
$$\codim~N\leq \codim~Q\leq C_{P}=2(p+q)-13/2,$$
which yields a contradiction as above.

For the oriented real Grassmannian $P=\Gr_{\mathbb{R}}(p, n)$, note that P has a polar-meridian pair of the form $(z, P)$ (see for example \cite[Proposition~1.9]{Nagano}), for some $z\in P$.
Therefore, one cannot immediately argue as before and hence needs to find a suitable polar-meridian pair. To this end, we show that $Q$ has a pole which is not mapped to a pole in $P$. Let $o$ be a base point for both $Q$ and $P$.
Since the oriented Grassmannian $Q_{1}$ in \eqref{Eq_Decomposition_Q} has a pole, by \cite[Proposition~6.5]{ChenNagano}, $Q$ has at least two poles
  namely say $p_{1}$ and $p_{2}$. Then by Lemma~\ref{L_Inclusion_Pairs} $\{p_{1}\}=Q^{o}_{+}(p_{1})\subseteq P^{o}_{+}(p_{1})$ and $\{p_{2}\}=Q^{o}_{+}(p_{2})\subseteq P^{o}_{+}(p_{2})$.
 Since $P$ has a unique pole, we conclude that at least one of the polars $P^{o}_{+}(p_{1})$ or $P^{o}_{+}(p_{2})$ is not a pole.
Assume that $P^{o}_{+}(p_{1})$ is not a pole.
By Lemma~\ref{L_Inclusion_Pairs} we have that $Q=Q^{o}_{-}(p_{1})\subseteq P^{o}_{-}(p_{1})\neq P$, where the last inequality follows from the assumption that $P^{o}_{+}(p_{1})$ is not a pole. Now we can proceed as for the complex and quaternionic Grassmannian and conclude that $Q$ has the same Cartan type as $P$.
\end{proof}

%%%%%%%%%%%%%%%%%%%%%%%%%%%%%%%%%% Bibliography %%%%%%%%%%%%%%%%%%%%%%%%%%%%%%%%%%%

%\bibliography{lib}{}
%\bibliography{index_symmetric_spaces04}
%\bibliographystyle{abbrv}

\def\cprime{$'$}

%\newpage

\pagebreak

\

\vfill

\begin{center}
\noindent
\begin{minipage}{\linewidth}
\small \noindent \textsc
{Manuel Amann} \\
\textsc{Institut f\"ur Mathematik}\\
\textsc{Differentialgeometrie}\\
\textsc{Universit\"at Augsburg}\\
\textsc{Universit\"atsstra\ss{}e 14 }\\
\textsc{86159 Augsburg}\\
\textsc{Germany}\\
[1ex]
\footnotesize
\textsf{manuel.amann@math.uni-augsburg.de}\\
\textsf{www.uni-augsburg.de/de/fakultaet/mntf/math/prof/diff/team/dr-habil-manuel-amann/}
\end{minipage}
\end{center}

\vspace{1cm}

\begin{center}
\noindent
\begin{minipage}{\linewidth}
\small \noindent \textsc
{Peter Quast} \\
\textsc{Institut f\"ur Mathematik}\\
\textsc{Differentialgeometrie}\\
\textsc{Universit\"at Augsburg}\\
\textsc{Universit\"atsstra\ss{}e 14 }\\
\textsc{86159 Augsburg}\\
\textsc{Germany}\\
[1ex]
\footnotesize
\textsf{peter.quast@math.uni-augsburg.de}\\
\end{minipage}
\end{center}

\vspace{1cm}

\begin{center}
\noindent
\begin{minipage}{\linewidth}
\small \noindent \textsc
{Masoumeh Zarei} \\
\textsc{Mathematisches Institut}\\
\textsc{Universität Münster}\\
\textsc{Einsteinstr. 62}\\
\textsc{48149 Münster }\\
%\textsc{86159 Augsburg}\\
\textsc{Germany}\\
[1ex]
\footnotesize
\textsf{masoumeh.zarei@uni-muenster.de}\\
\end{minipage}
\end{center}
\end{document}